\PassOptionsToPackage{dvipsnames,svgnames}{xcolor}
\documentclass{article} 
\usepackage{iclr2027_conference,times}

\usepackage{amsmath,amsfonts,bm}

\def\eqref#1{equation~\ref{#1}}

\def\1{\bm{1}}

\DeclareMathAlphabet{\mathsfit}{\encodingdefault}{\sfdefault}{m}{sl}
\SetMathAlphabet{\mathsfit}{bold}{\encodingdefault}{\sfdefault}{bx}{n}

\DeclareMathOperator*{\argmax}{arg\,max}

\makeatletter
\renewcommand{\eqref}[1]{\textup{\tagform@{\ref{#1}}}}  
\makeatother
   
\usepackage{hyperref}
\usepackage{url}

\usepackage[utf8]{inputenc}
\usepackage[T1]{fontenc}
\usepackage{hyperref}
\usepackage{url}
\usepackage{booktabs}
\usepackage{amsfonts}
\usepackage{amsmath}
\usepackage{amssymb}
\usepackage{amsthm}
\newtheorem{proposition}{Proposition}
\newtheorem*{propositionrestate}{Proposition \ref*{prop:feasibility} (restated)}
\usepackage{nicefrac}
\usepackage{microtype}
\usepackage{xcolor}
\usepackage{graphicx}
\usepackage{multirow}
\usepackage{algorithm}
\usepackage[noend]{algorithmic}
\usepackage{tikz}
\usepackage{subcaption}
\usepackage{enumitem}
\usetikzlibrary{shapes.geometric,arrows.meta,decorations.pathreplacing,positioning,calc}

\newcommand{\INPUT}{\REQUIRE}
\newcommand{\OUTPUT}{\ENSURE}

\usepackage{colortbl}
\usepackage{pgfplots}
\usepgfplotslibrary{groupplots,statistics}
\usetikzlibrary{calc}
\pgfplotsset{compat=1.18}

\providecommand{\argmax}{\operatorname*{argmax}}

\title{Optimization for One, Feasibility for Most: Decoupling Optimization from Feasibility in Chance-Constrained Programs}

\iclrfinalcopy

\author{
  Justin Dumouchelle\thanks{Corresponding author: \href{justin.dumouchelle@ucalgray.ca}{justin.dumouchelle@ucalgary.ca}} \\
  University of Calgary 
  \And
  Shunyu Yao \\
  University of Iowa 
   \And
  Beste Basciftci \\
  University of Iowa  
  \And
  Thiago Serra \\
  University of Iowa 
}

\newcommand{\methodname}{ACCESS}
\newcommand{\method}{\methodname\xspace}

\newcommand{\methoddelta}{\methodname($\delta$)\xspace}
\usepackage{xspace}

\begin{document}


\maketitle
\lhead{Preprint.}

\begin{abstract}

Chance-constrained stochastic programs (CCSPs) require solutions that 
satisfy uncertain constraints with high probability. 
Standard approaches either embed scenarios directly into optimization by considering sample average approximation (SAA)-based reformulations, 
growing prohibitively expensive and often failing to find feasible solutions; 
or employ conservative reformulations that guarantee feasibility but sacrifice objective quality by design. 
We propose \method (Adaptive Chance-Constrained Exploration via Structured Sampling), 
a structured sampling framework that decouples optimization from feasibility evaluation.
First, high-quality candidate solutions are generated by
solving many inexpensive single-scenario subproblems,
avoiding sampling directly in the constrained high-dimensional solution space.
Second,
feasibility is assessed separately via large-scale batched simulation over a massive number of scenarios, which is naturally amenable to GPU parallelism.
That reduces the search from a high-dimensional solution space to a two-dimensional parameter space controlling \emph{conservatism} and \emph{scenario diversity}; 
while concentrating on candidates near the feasibility boundary, where high-quality feasible solutions lie. 
We also introduce a variant, \methoddelta, which provides certified feasibility guarantees with probability $1{-}\delta$.
Experiments across three benchmark classes and five scenario distributions show that \method matches
or exceeds the best objective found by any method in nearly every setting, while \methoddelta is
empirically feasible on every instance. The conservative reformulations remain several percent worse in
objective, and SAA is frequently infeasible while taking an order of magnitude longer on the harder
benchmarks.

\end{abstract}

\section{Introduction}
\label{sec:intro}

Real-world decision-making under uncertainty routinely demands solutions that
are reliable with high probability. For instance, a production planner must ensure supply
meets demand in at least $90\%$ of scenarios; an energy grid operator must
maintain feasibility across $95\%$ of demand realizations; a logistics provider
must meet delivery commitments in all but rare cases \citep{shapiro2009lectures, kuccukyavuz2022chance}. 
These settings are
naturally modeled as \emph{chance-constrained stochastic programs} (CCSPs), which require finding a solution $\mathbf{x}$ from a set $\mathcal{X} \subseteq \mathbb{R}^n$ that satisfies a constraint with uncertain parameters with confidence level $\alpha \in (0,1)$; 
which is chosen as a large value by risk-averse decision makers, e.g., $\alpha \geq 0.90$. 
CCSPs are computationally intractable 
as their feasible regions are nonconvex and feasibility is expensive to evaluate except for special cases. They become especially challenging when the decisions become combinatorial, i.e., 
$\mathbf x \in \mathcal{X} \subseteq \{0,1\}^n$ (binary) or $\mathbf x \in \mathcal{X} \subseteq \mathbb{Z}^n$  
(integer). 
Hence, we evaluate feasibility empirically and refer to a solution as \emph{$\alpha$-feasible} if it satisfies all uncertain constraints in at least an $\alpha$ fraction of the evaluation scenarios. 

The dominant approach, 
\emph{sample average approximation} (SAA),
replaces the probability in the chance constraint with an empirical estimator based on an independently and identically distributed (i.i.d.) set of scenarios 
and solves a monolithic large-scale optimization problem consisting of a mixed-integer linear program (MILP) as its approximate representation. SAA enjoys strong
theoretical guarantees \citep{luedtke2008sample, Pagnoncelli2009}, but its computational cost
grows rapidly with the size of the sample set 
and the problem size. Reliably estimating a
99th-percentile constraint requires thousands of scenarios for discerning between feasible and infeasible solutions around the optimal value, making the
corresponding MILP intractable for combinatorial problems of practical interest.
An alternative class of methods, safe approximations, replaces the chance constraint with a tractable convex surrogate that guarantees feasibility by construction~\citep{rockafellar2000optimization,nemirovski2006convex, chen2007robust}. However, this conservatism comes at a cost as solutions are often suboptimal.
Both SAA and safe approximations are \emph{single-shot} methods, i.e., they solve a single optimization problem and thus produce a single solution.  

We take a fundamentally different approach.  Rather than solving once, we \emph{sample} many candidate solutions at varying conservatism levels and select the best feasible one.
Our framework, \method (Adaptive Chance-Constrained Exploration via Structured Sampling), decouples optimization from feasibility evaluation.
Rather than embedding scenarios into a monolithic optimization problem, we solve inexpensive single-scenario subproblems parameterized by two scalars: (1) a conservatism level $\gamma$ that controls how tightly constraints are enforced, and (2) a scenario magnitude $\beta$ that interpolates between nominal (deterministic) and diverse sampled scenarios.
Each $(\gamma, \beta)$ pair yields a candidate solution and feasibility is then assessed separately via batched evaluation over a large validation set.
Bayesian Optimization (BO) navigates this 2D parameter space to concentrate candidates near the feasibility boundary, where high-quality feasible solutions lie.
This approach exploits parallelism at two levels: many small subproblems can be solved independently across CPU cores, and feasibility evaluation over millions of scenarios parallelizes efficiently on GPUs for rapid best-of-$N$ selection.

{Figure~\ref{fig:sampling_comparison} illustrates this on instances of the chance-constrained knapsack problem.
In (a), SAA reaches high objective values but often violates the chance constraint whereas
conservative methods are feasible but lower in objective, whereas \method bridges the gap between them. In (b), the
candidates produced by random $(\gamma,\beta)$ pairs trace the trade-off between solution feasibility and
quality, with the search in \method\, concentrated on candidates where that trade-off meets $\alpha$-feasibility.}

\begin{figure}[h]
  \centering
  \includegraphics[width=\textwidth]{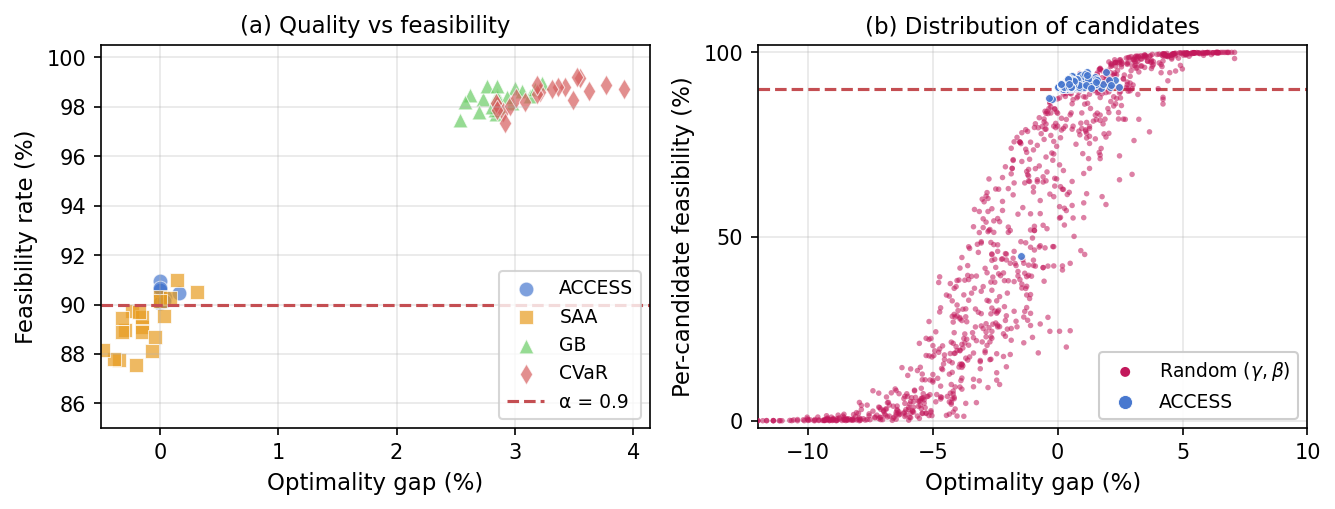}
  \caption{\textbf{(a)} Optimality gap versus feasibility rate on the largest knapsack instances at $\alpha{=}0.90$, one point
    per instance.  Gap is relative to the best $\alpha$-feasible objective found by any method, so a
    solution that beats it while violating the chance constraint appears at a negative gap. GB and CVaR denote the Gaussian--Bonferroni and CVaR safe approximations, respectively; all baselines are introduced in  Section~\ref{sec:results} and Appendix~\ref{app:baselines} 
    {\textbf{(b)} Feasibility and optimality gap of $1{,}000$ candidates from random $(\gamma,\beta)$
pairs and of the $1{,}000$ \method candidates, on a single knapsack instance.} 
    }
  \label{fig:sampling_comparison}
\end{figure}

\vspace{-0.2cm}
\textbf{Contributions.}
\begin{enumerate}[leftmargin=*]
\vspace{-0.2cm}
  \item \textbf{A two-parameter subproblem family, searched with BO.}
    We define a family of single-scenario MILP subproblems parameterized
    by a tightening parameter $\gamma$ and a diversification parameter $\beta$, 
    whose endpoints are the nominal scenario used
    in robust approximations ($\beta{=}0$) 
    and each realized scenario ($\beta{=}1$). 
    For each
    scenario, constrained BO searches the space $(\gamma, \beta)$, leaving the combinatorial structure
    to the solver and concentrating candidates near the feasibility boundary.

    \item \textbf{Distribution-free feasibility guarantees.} In classic scenario-based methods, e.g., SAA/CVaR, the scenarios that back the feasibility guarantee sit inside the optimization model, so their number is limited by what the solver can handle. \method\ moves the guarantee to the selection step: candidates are generated without the scenarios on which they are selected, and checking a candidate is cheap and amenable to GPU parallelism, providing strong guarantees that can be computed efficiently.  In addition, this yields a certified variant, \methoddelta, that returns an $\alpha$-feasible solution with probability at least $1-\delta$ for any distribution that can be sampled.

  \item \textbf{Comprehensive evaluation.}
    Across knapsack, matching and shortest-path benchmarks under five distributions, \method and
    \methoddelta return $\alpha$-feasible solutions at a worst-case mean gap of $0.04\%$ and $0.64\%$ to the
    best objective found by any method, while SAA is frequently infeasible and the safe
    approximations are overly conservative.  Moreover, solutions are found in a fraction of the time
    SAA requires, demonstrating the computational efficiency of the proposed approach.
\end{enumerate}
\vspace{-0.2cm}

\section{Preliminaries} 
\label{sec:prelim}

\subsection{Chance-Constrained Stochastic Programs}

A \emph{chance-constrained stochastic program} (CCSP) takes the form
\begin{subequations}
\begin{align}
  \max_{\mathbf{x} \in \mathcal{X}} \quad & f(\mathbf{x}) \label{eq:ccsp-obj} \\
  \text{s.t.} \quad & \text{Pr}_{\boldsymbol{\xi}}(\mathbf{x} \in \mathcal{P}(\boldsymbol{\xi})) \geq \alpha, %
  \label{eq:cc}
\end{align} \label{eq:ccsp}
\end{subequations}
where $\mathcal{X} \subseteq \mathbb{Z}^p \times \mathbb{R}^{n-p}$ is a deterministic feasible set on integer and continuous decision variables,
$f : \mathcal{X} \rightarrow \mathbb{R}$ is the objective function, 
and the chance constraint represents the set of feasible solutions $\mathcal{P}(\boldsymbol{\xi})$ under a random vector $\boldsymbol{\xi}$ with probability distribution $\mathbb{P}$ and confidence level $\alpha \in (0,1)$. 
In this work, 
we mainly focus on linear chance constraints by considering a polyhedral representation $\mathcal{P}(\boldsymbol{\xi}) := \{\mathbf{x}:\mathbf{A}(\boldsymbol{\xi})\mathbf{x} \le \mathbf{b}(\boldsymbol{\xi})\}$, where $\mathbf{A}(\boldsymbol{\xi}) \in \mathbb{R}^{d \times n}$ is a coefficient matrix, 
in which the $i$-th row is $\mathbf{a}_i(\boldsymbol{\xi}) \in \mathbb{R}^n$, 
and $\mathbf{b}(\boldsymbol{\xi}) \in \mathbb{R}^d$ is the vector of right-hand sides. 
For the nominal data, 
we use $\bar{\mathbf{A}} = \mathbb{E}_{\boldsymbol{\xi}}[\mathbf{A}(\boldsymbol{\xi})]$, 
in which the $i$-th row is $\bar{\mathbf{a}}_i$, 
and $\bar{\mathbf{b}} = \mathbb{E}_{\boldsymbol{\xi}}[\mathbf{b}(\boldsymbol{\xi})]$. 
This setting covers a broad class of optimization problems, such as knapsack, matching, and routing;  
and allows us to consider both individual  ($d=1$) and joint chance constraints ($d>1$), 
with all $d$ constraints holding \emph{simultaneously}.

\textbf{Solution approaches.}
In the widely used Sample Average Approximation (SAA), the probability over $\boldsymbol{\xi}$ is replaced with an empirical estimate over a set of scenarios 
$\mathcal{S} :=\{\boldsymbol{\xi}^{(i)}\}_{i=1}^{|\mathcal{S}|} \overset{\mathrm{i.i.d.}}{\sim} \mathbb{P}$, yielding
\begin{equation}
\small
\max_{\mathbf{x} \in \mathcal{X}}
\left\{
f(\mathbf{x}) :
\frac{1}{|\mathcal{S}|}\sum_{i=1}^{|\mathcal{S}|}
\mathbb{I}\left\{\mathbf{A}(\boldsymbol{\xi}^{(i)})\mathbf{x} \le \mathbf{b}(\boldsymbol{\xi}^{(i)})\right\}
\ge \alpha
\right\}.
\label{eq:general-saa}
\end{equation}
The indicator $\mathbb{I}$ equals one when $\mathbf{x}$ satisfies all constraints in a given scenario 
and zero otherwise. 
Hence, the SAA formulation requires simultaneous satisfaction of all constraints in at least $\lceil\alpha|\mathcal{S}|\rceil$ sampled scenarios. 
For the linear models considered in this paper, the polyhedral structure of 
$\mathbf{A}(\boldsymbol{\xi}^{(i)})\mathbf{x} \le \mathbf{b}(\boldsymbol{\xi}^{(i)})$ 
allows \eqref{eq:general-saa} to be reformulated as an equivalent MILP using binary variables and valid big-$M$ constraints, as detailed in Appendix~\ref{app:saa_formulation}.

Safe approximations replace the chance constraint~\eqref{eq:cc} with deterministic sufficient conditions, often expressed as convex constraints, yielding
\begin{equation}
\max_{\mathbf{x} \in \mathcal{X}}
\left\{ f(\mathbf{x}) : \mathbf{x} \in \mathcal{X}_{\mathrm{safe}} \right\},
\label{eq:general-safe}
\end{equation}
where $
\mathcal{X}_{\mathrm{safe}}
\subseteq
\left\{
\mathbf{x} \in \mathcal{X} :
\text{Pr}_{\boldsymbol{\xi}}(\mathbf{x} \in \mathcal{P}(\boldsymbol{\xi})) \geq \alpha
\right\}.$
Thus, every feasible solution of the safe approximation is feasible for the original problem, 
but some feasible solutions of the original problem may be excluded ($x \in \mathcal{X} :\text{Pr}_{\boldsymbol{\xi}}(\mathbf{x} \in \mathcal{P}(\boldsymbol{\xi})) \geq \alpha \not\Rightarrow x \in \mathcal{X}_{\mathrm{safe}}$); thereby affecting the quality of the solution obtained. 
Examples of safe approximations include CVaR and Bernstein approximations, as well as approaches that combine individual safe approximations through Bonferroni inequalities. We use two representative safe approximations, Gaussian--Bonferroni (GB) and CVaR, as benchmarks for our method. Details are provided in Section~\ref{sec:overview} and Appendix~\ref{app:safe_approx}.

\textbf{Evaluating feasibility.}
Computing the exact probability in~\eqref{eq:cc} is generally intractable. 
Hence, we use an empirical feasibility rate defined  
over a set of scenarios 
$\mathcal{T} :=\{\boldsymbol{\xi}^{(i)}\}_{i=1}^{|\mathcal{T}|} \overset{\mathrm{i.i.d.}}{\sim} \mathbb{P}$ 
as
\begin{equation}
p_\mathcal{T}(\mathbf{x}) = \frac{1}{|\mathcal{T}|}\sum_{i=1}^{|\mathcal{T}|} \mathbb{I}\left\{\mathbf{A}(\boldsymbol{\xi}^{(i)})\mathbf{x} \le \mathbf{b}(\boldsymbol{\xi}^{(i)})\right\},
\label{eq:emp-cc}
\end{equation}
and consider $\mathbf{x}$ to be \textit{$\alpha$-feasible} if $p_\mathcal{T}(\mathbf{x}) \geq \alpha$. 
This is standard practice in stochastic optimization. 

We note that this is the same approach used in SAA, 
but sets $\mathcal{S}$ and $\mathcal{T}$ are used for different purposes and at different times: 
the first to find a solution and the second to evaluate if the solution is indeed feasible. 
Drawing a parallel with machine learning tasks, $\mathcal{S}$ and $\mathcal{T}$ are akin to training and test sets. 

{Throughout the paper, multiple scenario sets such as $\mathcal{S}$ and $\mathcal{T}$ appear.  Each is an independent i.i.d.\ sample from $\mathbb{P}$, 
but to simplify notation we uniformly refer to their elements as $\xi^{(1)}, \xi^{(2)}, $ and so on.}

\subsection{Related Work}
\label{sec:review}

\textbf{Optimization under uncertainty.}
Chance-constrained programs are challenging to solve due to expensive feasibility verification and nonconvex, possibly disconnected feasible regions~\citep{nemirovski2006convex,ahmed2013probabilistic,kuccukyavuz2022chance}.
Two main approaches exist: convex conservative approximations~\citep{nemirovski2006convex,chen2007robust,rockafellar2000optimization} and SAA~\citep{luedtke2008sample}.
Both are single-shot methods with fixed conservatism levels.
We build on these quantile-type approximations but treat their conservatism constant as searchable, enabling empirical navigation of the feasibility-optimality tradeoff.
Our parameterized search also relates to robust optimization~\citep{ben2009robust,bertsimas2004price}, where a scalar controls uncertainty-set size.
Closest to us, \citet{li2015optimal} and \citet{yuan2017robust} also tighten the nominal
formulation and search for the margin rather than fixing it with a bound, differing in how that search
is carried out.  Both return a single solution, whereas we sample the scenario that defines each
subproblem and certify the one we select.

\textbf{Learning to optimize.}
Learning-based methods for combinatorial optimization have largely focused on
improving optimization algorithms, for example through learned branching
policies~\citep{khalil2016learning,gasse2019} or neural solution
construction methods~\citep{kool2019attention}. More closely related to our work are recent sampling-based approaches, e.g.,
\citep{sun2023difusco,sun2023revisiting,ichikawa2025optimization}, which generate
diverse candidate solutions through stochastic search dynamics directly in the combinatorial solution space. 
Constraints enter these methods as penalties in the sampling dynamics, so $\alpha$-feasibility would require a differentiable surrogate for the empirical feasibility rate~\eqref{eq:emp-cc}.  They also target problems where the MILP solver itself is the bottleneck, whereas in a CCSP the bottleneck arises from coupling the scenarios in one model, as in SAA, and solving many deterministic subproblems independently is comparatively cheap.

\textbf{Constraint learning.}
A related direction embeds learned predictive models directly into
MILP formulations~\citep{maragno2025mixed}. Building on
this paradigm,~\citet{alcantara2023ccl} propose chance-constraint learning via
conditional quantile models, assuming access to data from the underlying
stochastic process. Learn-and-embed methods for stochastic and bilevel
optimization~\citep{dumouchelle2022neur2sp,dumouchelle2024neur2bilo} similarly
train value-function surrogates and embed them in optimization models, with
training data generated by sampling decisions randomly. Recent work addresses the
resulting data-collection issue by alternating surrogate optimization and
retraining on samples near incumbent solutions~\citep{kronqvist2023alternating}.
In contrast, our framework does not learn an embedded surrogate or use iterative
active data collection; instead, it directly generates boundary-focused
candidate solutions through structured problem-space sampling. 
This distinction matters in combinatorial CCSPs since the training data must be generated, and random
solutions rarely fall near the feasibility boundary (Appendix~\ref{app:feas_distributions}), so a
learned constraint would first need a boundary-focused sampler, which is the problem \method
addresses.  

\section{Methodology}
\label{sec:method}

\method decouples optimization from feasibility evaluation: instead of solving one analytically
tightened problem, it \emph{searches} empirically over subproblems of varying conservatism ($\gamma$)
and diversity ($\beta$), as outlined in Figure~\ref{fig:pipeline}. 
In what follows, we describe the subproblem family (Section~\ref{sec:overview}), searching over it (Section~\ref{sec:boccp}), batched evaluation (Section~\ref{sec:oracle}), and a feasibility guarantee (Section~\ref{sec:guarantees}).

\begin{figure}[h!]
  \centering
  \includegraphics[width=\textwidth]{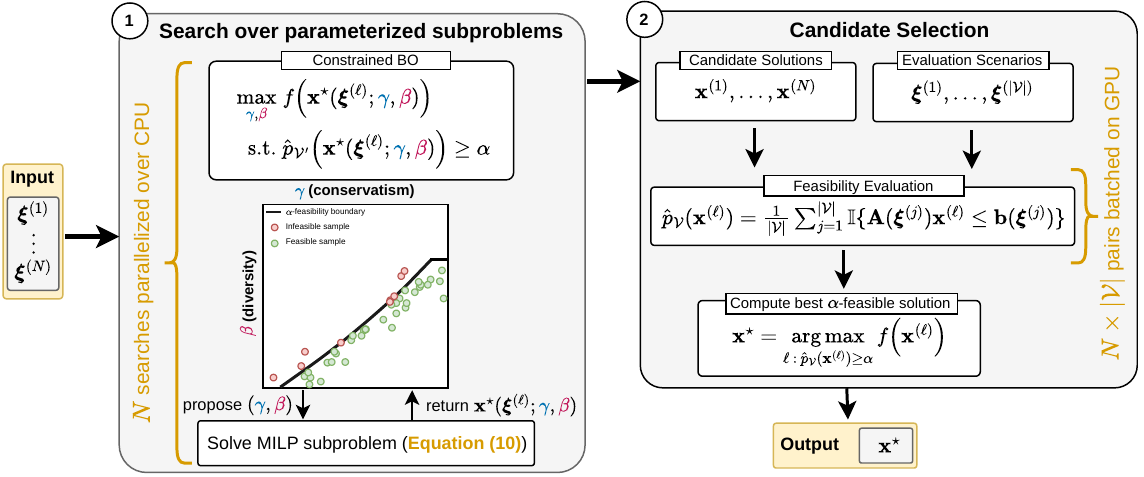}
  \caption{
The \method pipeline. 
(1) For each of the $N$ sampled scenarios, constrained BO searches over $(\gamma,\beta)$, in parallel on CPU.
(2) The resulting candidates $\mathbf{x}^{(1)},\dots,\mathbf{x}^{(N)}$ are scored against a set of validation scenarios $\mathcal{V}$ in one batched GPU pass, and the best $\alpha$-feasible solution is returned. \vspace{-1em}
  }
  \label{fig:pipeline}
\end{figure}

\subsection{Parameterized Subproblems}\label{sec:overview}

\method solves parameterized subproblems resembling two standard surrogates of the CCSP, the single-scenario problem and the safe approximation, to balance their opposite drawbacks.

\textbf{Single-scenario subproblems (and SAA).}
A natural heuristic is to sample a scenario $\boldsymbol{\xi} \sim \mathbb{P}$ and solve the deterministic problem induced by the realization of this scenario, i.e., 
\begin{equation}
\mathbf{x}^\star(\boldsymbol{\xi}) \in \argmax_{\mathbf{x} \in \mathcal{X}} \left\{
 f(\mathbf{x})
 \;:\; \mathbf{A}(\boldsymbol{\xi})\mathbf{x} \le \mathbf{b}(\boldsymbol{\xi})
\right\}.
\label{eq:subproblem_naive}
\end{equation}
Each deterministic solve is cheap and each draw $\boldsymbol{\xi}$ gives a different solution $\mathbf{x}^*(\boldsymbol{\xi})$, but this  solution is tailored to its scenario: 
some constraints might be binding ($\mathbf{a}_i(\mathbf{\xi})^T \mathbf{x}^*(\boldsymbol{\xi}) = b_i$ for some $i$), 
especially if the decision variables are continuous, 
and those likely violate the chance constraint under $\mathbb{P}$.
This is addressed in SAA by ensuring that the solution is feasible across a ratio $\alpha$ of a set of scenarios, 
but theoretical guarantees require a number of scenarios that is a computationally prohibitive.

\textbf{Safe approximations.}
Safe approximations instead solve a conservative formulation over the nominal problem.  
{Under multiplicative uncertainty, 
let us consider the component-wise setting in which 
$a_{i,j}(\boldsymbol{\xi}) = \bar{a}_{i,j}(1 + \sigma_i \xi_j)$, 
where $\bar{a}_{i,j} = E_{\boldsymbol{\xi}}[a_{i,j}(\boldsymbol{\xi})]$ 
and $\sigma_i$ is the relative standard deviation shared by all coefficients of row $i$. 
Using the element-wise (Hadamard) product, 
we denote the left-hand side of each constraint $i$ as 
$\mathbf{a}_i(\boldsymbol{\xi}) = \bar{\mathbf{a}}_i \odot
(\mathbf{1} + \sigma_i \boldsymbol{\xi})$, 
where $\mathbf{1}$ denotes the all-ones vector and $\boldsymbol{\xi} \in \mathbb{R}^n$ carries a noise
component per decision variable. 
For a fixed $\mathbf{x} = \mathbf{\tilde{x}}$,
the left-hand side of constraint $i$ is a random variable with mean
$\bar{\mathbf{a}}_i^\top \mathbf{\tilde{x}}$ 
and standard deviation $\sigma_i \|\bar{\mathbf{a}}_i \odot \mathbf{\tilde{x}}\|_2$. 
In the safe approximation, we solve for $\textbf{x}$ subject to
}
\begin{equation}
\bar{\mathbf{a}}_i^\top \mathbf{x} + \kappa\, \sigma_i \|\bar{\mathbf{a}}_i \odot \mathbf{x}\|_2 \leq \bar{b}_i,
\qquad \forall i = 1, \ldots, d,
\label{eq:gb_method}
\end{equation}
where $\kappa$ is selected to ensure feasibility under the scenario distribution $\mathbb{P}$ 
(see Appendix~\ref{app:safe_approx}). 
This approach is elegant for relying on a single parameter $\kappa$ 
thanks to $\sigma_i$ allowing us to scale the margin across constraints, 
but it is bound to yield suboptimal solutions because it is excessively conservative.

{\textbf{Striking a balance.} 
\methodname\, is derived from the approaches above in three steps. 
First, we circumvent the caveat of the safe approximation by working with an empirical parameter $\gamma$ replacing $\kappa$, 
meaning that we will search for the least conservative $\gamma$ whose solution remain $\alpha$-feasible on a large set of sampled scenarios. 
Because the norm in~\eqref{eq:gb_method} makes the subproblem a mixed-integer second-order cone program, 
we mitigate the cost of solving this subproblem multiple times by using the $L1$-norm instead. 
For simplicity, we assume $\mathbf{A}$ and $\mathbf{x}$ are non-negative, 
as they are in our benchmarks and in packing and covering problems generally. 
Under this assumption, 
$\|\bar{\mathbf{a}}_i \odot \mathbf{x}\|_2 \leq \|\bar{\mathbf{a}}_i \odot
\mathbf{x}\|_1 = \bar{\mathbf{a}}_i^\top \mathbf{x}$, 
 so replacing the norm by the linear term only tightens the constraint
and leaves a linear capacity, $(1 + \kappa\sigma_i)\,\bar{\mathbf{a}}_i^\top \mathbf{x}
\leq \bar{b}_i$, which is still a valid safe approximation in which 
\begin{equation}
(1 + \gamma\sigma_i)\,\bar{\mathbf{a}}_i^\top \mathbf{x} \leq \bar{b}_i
\qquad\text{or, equivalently,}\qquad
\bar{\mathbf{a}}_i^\top \mathbf{x} \leq \frac{\bar{b}_i}{1 + \gamma\sigma_i} \qquad \forall i = 1, \ldots, d.
\label{eq:l1_constraint}
\end{equation}

Second, 
we produce multiple solutions instead of a single one 
by solving the subproblem on a set of scenarios $\{\mathbf{\xi}^{(\ell)}\}_{\ell=1}^N$ 
instead of only the nominal case. 
For each scenario $\mathbf{\xi}$, the subproblem becomes 
\begin{equation}
\mathbf{x}^\star(\boldsymbol{\xi}; \gamma) \in \argmax_{\mathbf{x} \in \mathcal{X}} \left\{
f(\mathbf{x}) \;:\;
\mathbf{a}_i(\boldsymbol{\xi})^\top \mathbf{x} \leq \frac{b_i(\boldsymbol{\xi})}{1 + \gamma \hat{\sigma}_i} \> ~ ~ \forall i = 1,  \ldots, d
\right\},
\label{eq:subproblem_gamma}
\end{equation}
where $\hat{\sigma}_i$ is empirically estimated from sampled scenarios (see Appendix~\ref{app:sigma_estimate} for details). 
Similar to the single-scenario case, 
we obtain possibly different candidate solutions from different scenarios; 
whereas 
parameter $\gamma$ supplies the conservatism that the single-scenario case lacks 
and with feasibility defined on the joint constraints rather than through a union bound as in the safe approximation case.  
Moreover, 
keeping an estimate per row (constraint) rather than a single estimate across all rows improves solution quality when the rows carry
different noise levels, as measured in Appendix~\ref{app:het_ablation}. 

Third, 
we avoid excessive conservatism from imposing the scenario constraints  
by using an additional parameter $\beta$ to interpolate the constraints between the nominal problem and the realized scenario: 
\begin{equation}
\mathbf{a}_i(\boldsymbol\xi;\beta) = (1 - \beta)\,\bar{\mathbf{a}_i} + \beta\,\mathbf{a}_i(\boldsymbol{\xi}),
\qquad
{b}_i(\boldsymbol\xi;\beta) = (1 - \beta)\,\bar{{b}}_i + \beta\,{b}_i(\boldsymbol{\xi}),
\qquad \beta \in [0, 1].
\label{eq:rescaling}
\end{equation}
Hence, $\beta$
interpolates between the nominal constraints in~\eqref{eq:l1_constraint} ($\beta = 0$) and the realized scenario constraints in~\eqref{eq:subproblem_gamma} ($\beta = 1$), 
thereby spanning rather than extending either case. 
The mixture needs only that
$\boldsymbol{\xi}$ can be sampled and requires no further distributional assumptions. 
}

Therefore, the full parameterized subproblem in \methodname\, is formulated as
\begin{equation}
\mathbf{x}^\star(\boldsymbol{\xi}; \gamma, \beta) \in \argmax_{\mathbf{x} \in \mathcal{X}} \left\{
f(\mathbf{x}) \;:\;
{\mathbf{a}_i(\boldsymbol{\xi};\beta)}^\top \mathbf{x} \le \frac{b_i(\boldsymbol{\xi};\beta)}{1 + \gamma \hat{\sigma}_i} ~ ~ \forall i\in\{1,\ldots, d\}
\right\}.
\label{eq:2d_problem}
\end{equation}
In other words, 
we search over the two-dimensional \emph{parameter space} $(\gamma, \beta)$ 
instead of over the high-dimensional \emph{solution space} $\textbf{x}$. 
Parameter $\gamma$ controls \emph{conservatism}: a larger $\gamma$ tightens the capacity of the constraints. 
Parameter $\beta$ controls \emph{diversity}: a larger $\beta$ reshapes the constraints towards structurally different solutions. 
Searching over the parameter space leaves the combinatorial structure to the optimization solver: 
with $\gamma$ ordering the subproblems by conservatism and $\beta$ counterbalancing the spread across scenarios, 
the search concentrates candidates near the feasibility boundary.

\textbf{Assumptions and generalizations.}
The derivation assumed multiplicative independent perturbations and non-negative $\mathbf{a}_i$ and $\mathbf{x}$. 
These only motivate the parameterization: negative-valued data/decisions and additive rather than multiplicative
perturbations need only slight modifications to~\eqref{eq:2d_problem}, which stays linear in both cases.
The $\beta$ interpolation is natural for continuous distributions, but may require
alternative interpolation schemes for discrete distributions, e.g., Bernoulli.  
The guarantees of Section~\ref{sec:guarantees} require i.i.d.\ samples from $\mathbb{P}$ and nothing else, whereas safe approximations need distributional knowledge, e.g., a moment-generating-function bound for Bernstein, first and second moments for Cantelli~\citep{calafiore2006distributionally,nemirovski2006convex}). 

\subsection{Structured Sampling via Bayesian Optimization}\label{sec:boccp}

\method draws $N$ scenarios and runs an independent two-dimensional search for each one. Within a
scenario, the search maximizes the subproblem objective subject to maintaining $\alpha$-feasibility on a set 
of sample scenarios $\mathcal{V}'$. The set $\mathcal{V}'$ only scores the points within each
search, while the $N$ candidates the searches return are scored on a much larger set $\mathcal{V}$
that no search has seen, which is what allows the final selection to be certified
(Section~\ref{sec:guarantees}).

\textbf{Bayesian optimization.}
For each scenario, we model the $(\gamma, \beta) \to (\text{objective}, \text{feasibility})$
mapping as a constrained black-box problem and solve it with Gaussian process
surrogates~\citep{rasmussen2006gaussian} using constrained expected
improvement~\citep{balandat2020botorch, ament2023unexpected}, with a small set of quasi-random samples.
We search over $\beta\in[0,1]$ and $\gamma\in[0,\gamma_u]$, where
$\gamma_u$ is set per instance by solving the nominal problem over a coarse grid of
$\gamma$ and taking twice the smallest value whose solution is $\alpha$-feasible on $\mathcal{V}'$.  
Implementation details and the empirical behaviour of the search appear in
Appendices~\ref{app:hyperparameters} and
\ref{app:boccp_convergence}--\ref{app:boccp_param_analysis}.

\begin{algorithm}[h!]
\caption{\method: Adaptive Chance-Constrained Exploration via Structured Sampling}
\label{alg:access}
\begin{algorithmic}[1]
\footnotesize
\INPUT Distribution $\mathbb{P}$, confidence level $\alpha$, scenarios $N$, budget $N_{\text{init}} + N_{\text{iter}}$, scenario sets $\mathcal{V}', \mathcal{V}$
\OUTPUT $\alpha$-feasible solution $\mathbf{x}^*$
\STATE Draw scenarios $\{\boldsymbol{\xi}^{(\ell)}\}_{\ell=1}^{N} \sim \mathbb{P}$; $\;\mathcal{C} \leftarrow \emptyset$
\FOR{$j = 1, \ldots, N$ \textbf{in parallel}}
    \FOR{$t = 1, \ldots, N_{\text{init}} + N_{\text{iter}}$}
        \IF{$t \leq N_{\text{init}}$}
            \STATE $(\gamma_t, \beta_t) \leftarrow$ quasi-random sample from the $(\gamma, \beta)$ box
        \ELSE
            \STATE $(\gamma_t, \beta_t) \leftarrow$ constrained BO acquisition 
        \ENDIF
        \STATE Solve MILP \eqref{eq:2d_problem}: $\mathbf{x}_t \leftarrow \mathbf{x}^*(\boldsymbol{\xi}^{(\ell)}, \gamma_t, \beta_t)$ \hfill $\triangleright$ \textit{parametrized subproblems}
        \STATE Score $\hat{p}_{\mathcal{V}'}(\mathbf{x}_t)$; update BO
    \ENDFOR
    \STATE $\mathcal{C} \leftarrow \mathcal{C} \cup \arg\max\{f(\mathbf{x}) \;:\; \mathbf{x} \in \{\mathbf{x}_t\}^{N_{\text{init}} + N_{\text{iter}}}_{t=1},\; \hat{p}_{\mathcal{V}'}(\mathbf{x}) \geq \alpha\}$ \hfill $\triangleright$ \textit{best per scenario}
\ENDFOR
\STATE $\mathbf{x}^* \leftarrow \arg\max\{f(\mathbf{x}) \;:\; \mathbf{x} \in \mathcal{C},\; \hat{p}_{\mathcal{V}}(\mathbf{x}) \geq \alpha\}$ \hfill $\triangleright$ \textit{$\hat{p}_{\mathcal{V}}$ GPU-batched}
\RETURN $\mathbf{x}^*$
\end{algorithmic}
\end{algorithm}

\subsection{GPU-Accelerated Feasibility Evaluation}
\label{sec:oracle}

Feasibility evaluation decomposes across candidate–scenario pairs: each indicator $\mathbb{I}\{\mathbf{A}(\boldsymbol{\xi}^{(v)})\mathbf{x}^{(i)} \leq \mathbf{b}\}$ is independent.
We compute all $N \times |\mathcal{V}|$ indicators in a single batched GPU operation, enabling evaluation of $N = 1000$ candidates over $|\mathcal{V}| = 10^6$ scenarios in seconds.
This extends naturally to any risk measure that decomposes across scenarios, such as CVaR or variance-based objectives.
In two-stage stochastic programs, evaluating a candidate requires solving a recourse problem per scenario, generally precluding batched computation unless the recourse admits specific structure (e.g., dynamic programming representability~\citep{zhao2026from}).
Chance constraints require only indicator evaluation, making our setting trivially GPU-parallel, a key advantage of decoupling that our proposed pipeline leverages.

\subsection{Feasibility Guarantees}\label{sec:guarantees}

\method relies on empirical feasibility, which is only an estimate, whereas the true feasibility of the
returned solution can fall below estimates on $\mathcal{V}$. Decoupling optimization from
evaluation makes this shortfall easy to bound: since \method generates its $N$ candidates without
using $\mathcal{V}$, a union bound over those $N$ solutions suffices, and the shortfall is at most
a small $\epsilon_\delta$ with probability at least $1-\delta$, whichever candidate is selected.
Raising the selection threshold by $\epsilon_\delta$ therefore yields \methoddelta, whose returned
solution is feasible under the true distribution $\mathbb{P}$ with probability at least $1-\delta$.

\begin{proposition}\label{prop:feasibility}
Fix $\delta \in (0,1)$ and let $\hat{\mathbf{x}}_1, \ldots, \hat{\mathbf{x}}_N$ be generated
independently of $\mathcal{V}$. For each $j$, let $p_j = \Pr_{\boldsymbol{\xi}}(\mathbf{A}(\boldsymbol{\xi})\hat{\mathbf{x}}_j \leq \mathbf{b}(\boldsymbol{\xi}))$
be the true feasibility probability of $\hat{\mathbf{x}}_j$ and $\hat{p}_j$ its empirical rate on
$\mathcal{V}$. With probability at least $1 - \delta$ over the draw of $\mathcal{V}$,
\begin{equation}
p_j \geq \hat{p}_j - \sqrt{\frac{\ln(N/\delta)}{2|\mathcal{V}|}} \quad \text{simultaneously for every } j = 1, \ldots, N.
\label{eq:hoeffding_bound}
\end{equation}
\end{proposition}
The candidates satisfy the condition because the search scores them on $\mathcal{V}'$, which is drawn
independently from $\mathcal{V}$. The proof is based on Hoeffding's inequality~\citep{hoeffding1963} with a
union bound over the $N$ candidates (see Appendix~\ref{app:hoeffding} for details).

\textbf{Certified selection: \methoddelta.}
Proposition~\ref{prop:feasibility} turns the empirical rate into a certificate. Selecting any
candidate with
\begin{equation}
\hat{p}_{\mathcal{V}}(\mathbf{x}) \geq \alpha + \epsilon_\delta, \qquad
\epsilon_\delta = \sqrt{\frac{\ln(N/\delta)}{2|\mathcal{V}|}},
\label{eq:certified_threshold}
\end{equation}
guarantees $p(\mathbf{x}) \geq \alpha$ with probability at least $1 - \delta$. We call this
variant \methoddelta: the same selection rule with the threshold raised by $\epsilon_\delta$.

\textbf{Comparison with SAA.}
SAA also offers a similar style of feasibility guarantee.  However, it embeds its scenario set in the optimization problem, so its guarantee must hold over every
decision in the sample-approximation feasible region, a multiplicity that can be exponential for
integer programs~\citep{luedtke2008sample}. Correcting over $N$ pre-generated candidates instead lets us
evaluate on very large scenarios sets, i.e., $|\mathcal{V}| = 10^6$ scenarios while keeping optimization cheap.
For example, with $N = 1000$ and $\delta = 0.05$ we get $\epsilon_\delta = 0.0022$. The guarantee is distribution-free and needs no parametric
assumption on $\mathbb{P}$.

\section{Computational Results}
\label{sec:results}

\vspace{-0.5em}
\textbf{Computing setup.} 
We used a machine with an NVIDIA H100 GPU and an Intel 6972P CPU with 64GB of RAM.  All MILPs were solved on Gurobi 12.0.0~\citep{gurobi}; all neural network implementations and GPU computations use PyTorch 2.5.0~\citep{paszke2019pytorch}.

\textbf{Benchmarks.}
We use three problem classes: (1) \textit{Knapsack (KP)}, selecting from $n$ items to
maximize profit under $d$ random capacity constraints; (2) \textit{Maximum Weight Matching (MWM)},
a maximum-weight matching on an $n \times n$ bipartite graph under $d$ random edge-weight
constraints; and (3) \textit{Resource-Constrained Shortest Path (SRCSP)}, a shortest path through a
graph with $|A|$ arcs under $d$ random resource constraints.
Each has three sizes (Small, Medium, Large): $(n,d)$ or $(|A|,d) \in \{(20,3), (50,5), (100,10)\}$
for KP and SRCSP, and $n \in \{10,15,20\}$ with $d = 5$ for MWM.
Each has multiplicative uncertainty,
$\mathbf{a}_i(\boldsymbol{\xi}) = \bar{\mathbf{a}}_i \odot \max(\mathbf{1} + \sigma\boldsymbol{\xi},\, 0)$
with $\sigma = 0.2$ and one i.i.d.\ $\xi_j \sim \mathcal{N}(0,1)$ per decision variable, shared
across the $d$ rows.
We then focus on the non-Gaussian case for KP at $n \in \{50, 100\}$, drawing $\boldsymbol{\xi}$
from four distributions: (1) heavy tails (\emph{Student-$t_3$}), (2) rare large shocks
(\emph{Jump}), (3) two correlated blocks of the most attractive items (\emph{Clustered}), and
(4) a fixed random 30\% of items carrying triple the noise (\emph{Risky subset}).
We also explore heterogeneous uncertainty by varying $\boldsymbol\sigma$ across constraints in three further
configurations (Appendix~\ref{app:het_ablation}).
For every configuration, we consider 20 instances and confidence levels
$\alpha \in \{0.90, 0.95, 0.99\}$, resulting in a total of 1{,}200 instances.
Formulations, distribution parameters, and generation procedures are in
Appendix~\ref{app:benchmarks}.

\textbf{Methods.}
We evaluate two variants of our approach: \method and \methoddelta with $\delta = 0.05$, which
returns a solution that is $\alpha$-feasible with probability at least $1 - \delta$ over the draw
of $\mathcal{V}$.
The two differ only in the selection rule and share the same candidate pool, generated from
$N = 1000$ scenarios with $N_{\text{init}} = 5$ initialization points and $N_{\text{iter}} = 10$ BO
iterations per scenario.
We use $|\mathcal{V}^\prime| = 10^4$ scenarios to score points within the BO and
$|\mathcal{V}| = 10^6$ for final selection.
Every subproblem solve is capped at 3\,s on a single thread, although most solves terminate in
orders of magnitude less time, and we parallelize across scenarios with 40 workers.
Reported runtimes cover both phases, the CPU subproblem solves and the GPU feasibility evaluation
over $\mathcal{V}$, and Appendix~\ref{app:detailed_runtime} reports both separately.

\textbf{Baselines.}
We compare against (1) the big-M sample average approximation
(SAA)~\citep{luedtke2008sample}, (2) the sample CVaR approximation applied per row with a Bonferroni
split~\citep{rockafellar2000optimization,nemirovski2006convex},
and
(3) the Gaussian--Bonferroni (GB) safe approximation, which is exact for each row under Gaussian
$\boldsymbol{\xi}$~\citep[Section~2.6]{lobo1998applications} and conservative for the
joint constraint through the Bonferroni split.  GB does not apply to the non-Gaussian benchmarks and
is omitted there.
The main results use $|\mathcal{S}| = 1000$ scenarios for SAA and CVaR, drawn independently of
$\mathcal{V}$ and $\mathcal{T}$; results for $|\mathcal{S}| \in \{10, 100, 10000\}$ are in
Appendix~\ref{app:baseline_analysis}.
All baseline solves use Gurobi with 6 threads and a 3600\,s time limit, returning the best
incumbent at the limit.
Baseline formulations and details are in Appendix~\ref{app:baselines}.

\textbf{Evaluation metrics.}
We evaluate feasibility using $\mathcal{T}=10^6$ sampled scenarios per instance via~\eqref{eq:emp-cc} (separate from the selection set $\mathcal{V}$ used in \method).
We report runtime, $\alpha$-feasibility rate (fraction of instances where the method finds an $\alpha$-feasible solution), and two gap metrics.  A per-benchmark runtime breakdown is in Appendix~\ref{app:detailed_runtime}.
\emph{FO-gap} (Table~\ref{tab:results}) measures the gap to the Best Known Objective (BKO). On each instance, $\mathrm{FO}_m$ is the objective value of the solution method $m$ returns, and BKO is the largest $\mathrm{FO}_m$ over the methods whose returned solution is $\alpha$-feasible on $\mathcal{T}$.  $\text{FO-gap}_m = (\mathrm{BKO}-\mathrm{FO}_m)/\mathrm{BKO}$ is reported only when the solution of $m$ is itself $\alpha$-feasible, and averaged over those instances.
\emph{Optimality gap} (Figure~\ref{fig:main_distributions}) uses the same reference as FO-gap, the best $\alpha$-feasible objective found by any method, so figures and tables agree; every solution is plotted, feasible or not, and an infeasible solution that beats the reference appears at a negative gap. This shows the full feasibility-quality tradeoff without a second definition of the gap.

\textbf{Results.}
Table~\ref{tab:results} summarizes results across all benchmarks with Gaussian uncertainty.
\method attains the best or tied-best FO-gap in 26 of the 27 settings, within $0.04\%$ of the best
known objective throughout. 
\method falls below $\alpha$ on a small minority of instances; \methoddelta is $\alpha$-feasible on
every instance, at a worst case of $0.64\%$.
SAA's gaps are averaged only over the instances where it returns a feasible solution, a subset that
falls to $0\%$ on the two larger KP sizes at $\alpha{=}0.99$, so the comparison flatters it.
The safe approximations invert the trade-off: GB is $\alpha$-feasible on every instance and CVaR on
all but three settings, but they give up $1.6$--$7.5\%$ and $0.9$--$10.4\%$ of objective
respectively.
Per-instance distributions are in Appendix~\ref{app:instance_distributions}.

\vspace{-0.2em}
Figure~\ref{fig:main_distributions} evaluates all methods when $\boldsymbol{\xi}$
is not Gaussian, on KP at $n{=}100$ and $\alpha{=}0.99$.
The three methods separate the same way in every panel.
\method is $\alpha$-feasible on nearly all instances with mean gap of at most 0.32\%.
CVaR is $\alpha$-feasible on all 80 instances, but pays in solution quality.
SAA stays below $\alpha$ on 65\% of instances even at $|\mathcal{S}|{=}10^4$, and its negative
gaps are solutions that beat the best $\alpha$-feasible objective while violating the chance
constraint.  See Appendix~\ref{app:nongaussian} for details.

\begin{table}[t]
  \caption{
  Feasibility, objective, and runtime results. Each cell
reports FO-gap / $\alpha$-feasibility rate, followed by runtime in seconds
(averaged over $\alpha$ values).
FO-gap is relative to BKO (Best Known Objective): the best $\alpha$-feasible objective across all methods per instance.
SAA and CVaR use $|\mathcal{S}|{=}1000$ scenarios; \methoddelta applies Hoeffding correction with $\delta{=}0.05$.
Sizes: KP: $n \in \{20,50,100\}$, $d \in \{3,5,10\}$; MWM:
$10\times10/15\times15/20\times20$, $d=5$; SRCSP:
$|A| \in \{20,50,100\}$, $d \in \{3,5,10\}$.
Each entry is averaged over 20 test instances.
}
  \label{tab:results}
  \centering
  \resizebox{\textwidth}{!}{%
  \setlength\tabcolsep{3pt}
  \begin{tabular}{ll rrrr rrrr rrrr}
    \toprule
    & & \multicolumn{4}{c}{Small}
      & \multicolumn{4}{c}{Medium}
      & \multicolumn{4}{c}{Large} \\
    \cmidrule(lr){3-6}\cmidrule(lr){7-10}\cmidrule(lr){11-14}
    & Method
      & Time & $\alpha{=}.90$ & $\alpha{=}.95$ & $\alpha{=}.99$
      & Time & $\alpha{=}.90$ & $\alpha{=}.95$ & $\alpha{=}.99$
      & Time & $\alpha{=}.90$ & $\alpha{=}.95$ & $\alpha{=}.99$ \\
    \midrule
        \multirow{5}{*}{\textbf{KP}}
      & SAA & 574 & 0.08\%\,/\,75\% & 0.11\%\,/\,50\% & \textbf{0.00\%}\,/\,45\% & 3588 & 0.21\%\,/\,40\% & 0.19\%\,/\,25\% & ---\,/\,0\% & 3600 & 0.11\%\,/\,25\% & 0.29\%\,/\,5\% & ---\,/\,0\% \\
      & CVaR & 2.5 & 5.82\%\,/\,100\% & 4.88\%\,/\,100\% & 3.33\%\,/\,100\% & 7.3 & 4.16\%\,/\,100\% & 3.40\%\,/\,100\% & 1.37\%\,/\,85\% & 156 & 3.25\%\,/\,100\% & 2.45\%\,/\,100\% & 0.89\%\,/\,35\% \\
      & GB & 0.1 & 3.52\%\,/\,100\% & 2.94\%\,/\,100\% & 2.54\%\,/\,100\% & 1.5 & 3.00\%\,/\,100\% & 2.65\%\,/\,100\% & 2.14\%\,/\,100\% & 111 & 2.87\%\,/\,100\% & 2.51\%\,/\,100\% & 2.05\%\,/\,100\% \\
      \arrayrulecolor{black!30}\cmidrule{2-14}\arrayrulecolor{black}
      & \method & 88 & \textbf{0.00\%}\,/\,100\% & \textbf{0.01\%}\,/\,100\% & 0.00\%\,/\,100\% & 111 & \textbf{0.00\%}\,/\,95\% & \textbf{0.04\%}\,/\,100\% & \textbf{0.00\%}\,/\,100\% & 129 & \textbf{0.01\%}\,/\,100\% & \textbf{0.01\%}\,/\,90\% & \textbf{0.00\%}\,/\,95\% \\
      & \methoddelta & 88 & 0.07\%\,/\,100\% & 0.10\%\,/\,100\% & 0.50\%\,/\,100\% & 111 & 0.06\%\,/\,100\% & 0.11\%\,/\,100\% & 0.39\%\,/\,100\% & 129 & 0.03\%\,/\,100\% & 0.05\%\,/\,100\% & 0.27\%\,/\,100\% \\
    \midrule
    \multirow{5}{*}{\textbf{MWM}}
      & SAA & 1943 & 0.09\%\,/\,75\% & 0.17\%\,/\,75\% & 0.13\%\,/\,55\% & 3600 & 0.80\%\,/\,85\% & 0.62\%\,/\,55\% & 0.11\%\,/\,55\% & 3600 & 1.38\%\,/\,65\% & 0.84\%\,/\,80\% & 0.04\%\,/\,40\% \\
      & CVaR & 14 & 9.20\%\,/\,100\% & 6.76\%\,/\,100\% & 5.15\%\,/\,100\% & 286 & 5.75\%\,/\,100\% & 5.18\%\,/\,100\% & 3.39\%\,/\,95\% & 2372 & 4.49\%\,/\,100\% & 3.85\%\,/\,100\% & 1.94\%\,/\,100\% \\
      & GB & 1.9 & 5.71\%\,/\,100\% & 4.58\%\,/\,100\% & 4.07\%\,/\,100\% & 39 & 3.62\%\,/\,100\% & 3.04\%\,/\,100\% & 2.63\%\,/\,100\% & 761 & 2.68\%\,/\,100\% & 2.22\%\,/\,100\% & 1.63\%\,/\,100\% \\
      \arrayrulecolor{black!30}\cmidrule{2-14}\arrayrulecolor{black}
      & \method & 136 & \textbf{0.00\%}\,/\,100\% & \textbf{0.00\%}\,/\,100\% & \textbf{0.00\%}\,/\,100\% & 291 & \textbf{0.00\%}\,/\,100\% & \textbf{0.00\%}\,/\,95\% & \textbf{0.00\%}\,/\,100\% & 565 & \textbf{0.00\%}\,/\,100\% & \textbf{0.00\%}\,/\,100\% & \textbf{0.00\%}\,/\,95\% \\
      & \methoddelta & 136 & 0.00\%\,/\,100\% & 0.03\%\,/\,100\% & 0.11\%\,/\,100\% & 291 & 0.09\%\,/\,100\% & 0.06\%\,/\,100\% & 0.56\%\,/\,100\% & 565 & 0.11\%\,/\,100\% & 0.06\%\,/\,100\% & 0.45\%\,/\,100\% \\
    \midrule
    \multirow{5}{*}{\textbf{SRCSP}}
      & SAA & 1.1 & 1.84\%\,/\,95\% & \textbf{0.00\%}\,/\,100\% & \textbf{0.00\%}\,/\,95\% & 21 & \textbf{0.00\%}\,/\,90\% & 0.98\%\,/\,95\% & 0.79\%\,/\,70\% & 551 & 0.07\%\,/\,85\% & 0.28\%\,/\,75\% & 0.41\%\,/\,55\% \\
      & CVaR & 0.5 & 10.38\%\,/\,100\% & 7.16\%\,/\,100\% & 3.10\%\,/\,100\% & 2.9 & 8.56\%\,/\,100\% & 6.84\%\,/\,100\% & 6.95\%\,/\,100\% & 11 & 8.62\%\,/\,100\% & 7.44\%\,/\,100\% & 4.62\%\,/\,100\% \\
      & GB & 0.0 & 7.51\%\,/\,100\% & 3.29\%\,/\,100\% & 2.05\%\,/\,100\% & 0.1 & 5.77\%\,/\,100\% & 4.98\%\,/\,100\% & 4.30\%\,/\,100\% & 0.5 & 6.77\%\,/\,100\% & 6.25\%\,/\,100\% & 5.40\%\,/\,100\% \\
      \arrayrulecolor{black!30}\cmidrule{2-14}\arrayrulecolor{black}
      & \method & 81 & \textbf{0.00\%}\,/\,100\% & 0.00\%\,/\,100\% & 0.00\%\,/\,100\% & 91 & 0.03\%\,/\,100\% & \textbf{0.00\%}\,/\,100\% & \textbf{0.00\%}\,/\,100\% & 109 & \textbf{0.04\%}\,/\,95\% & \textbf{0.00\%}\,/\,100\% & \textbf{0.00\%}\,/\,100\% \\
      & \methoddelta & 81 & 0.00\%\,/\,100\% & 0.00\%\,/\,100\% & 0.00\%\,/\,100\% & 91 & 0.03\%\,/\,100\% & 0.24\%\,/\,100\% & 0.59\%\,/\,100\% & 109 & 0.10\%\,/\,100\% & 0.13\%\,/\,100\% & 0.64\%\,/\,100\% \\
    \bottomrule
  \end{tabular}}
\end{table}

\begin{figure}[t]
  \centering
  \includegraphics[width=\textwidth]{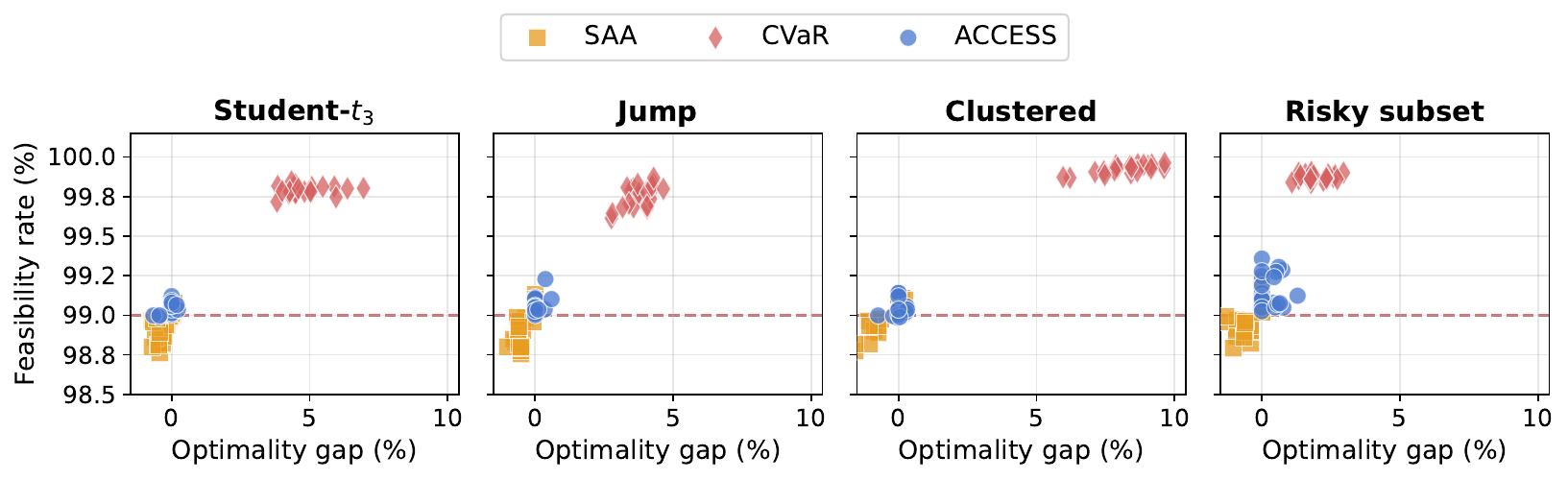}
  \caption{Feasibility rate vs.\ optimality gap under non-Gaussian distributions, on KP with
  $n{=}100, \alpha{=}0.99$. Each point is one of 20 instances, and the dashed line marks $\alpha$.
  SAA and CVaR use $|\mathcal{S}|{=}10^4$. 
  Gap is relative to the best $\alpha$-feasible objective, so infeasible solutions may have negative gaps. \vspace{-1em}
  }
  \label{fig:main_distributions}
\end{figure}

\textbf{Limitations.}
{
\method replaces one large solve by thousands of small independent ones, so its cost presumes a cheap deterministic subproblem and benefits heavily from parallel hardware.  
Its guarantee concerns feasibility rather than objective value, although empirically the returned solutions lie within a fraction of a percent of the best objective found by any method.}

\section{Conclusion}\label{sec:conclusion}

We introduced \method, a structured sampling approach for the computationally efficient and effective solution of chance-constrained stochastic programs by reducing the search from high-dimensional solution space to a 2D $(\gamma, \beta)$ parameter space while capturing the trade-off between optimality and feasibility. By concentrating candidates near the feasibility boundary via Bayesian optimization, \method consistently finds high-quality $\alpha$-feasible solutions while \methoddelta provides certified guarantees with probability $1-\delta$.
Empirically, \method attains near-zero optimality gap at high feasibility across all benchmarks,
where SAA is frequently infeasible and the safe approximations are uniformly conservative.
Future directions include extending to alternative probabilistic constraint structures and risk measures such as distributionally robust formulations,
generalizing to richer formulations that allow for more flexibility, and investigating alternative search strategies beyond Bayesian optimization.

\bibliography{refs}
\bibliographystyle{iclr2027_conference}

\clearpage
\appendix

\noindent\textbf{Appendix}
\begin{itemize}
  \item[\ref{app:benchmarks}] Benchmarks and Details \dotfill \pageref{app:benchmarks}
  \item[\ref{app:baselines}] Baselines \dotfill \pageref{app:baselines}
  \item[\ref{app:certified_guarantees}] Certified Feasibility Guarantees \dotfill \pageref{app:certified_guarantees}
  \item[\ref{app:boccp_analysis}] Algorithmic Details \dotfill \pageref{app:boccp_analysis}
  \item[\ref{app:nonstandard}] Non-Standard Uncertainty \dotfill \pageref{app:nonstandard}
  \item[\ref{app:boccp_gamma_only}] Ablation: Nominal-Only Tightening \dotfill \pageref{app:boccp_gamma_only}
  \item[\ref{app:baseline_analysis}] Baseline Analysis \dotfill \pageref{app:baseline_analysis}
  \item[\ref{app:instance_distributions}] Per-Instance Distributional Results \dotfill \pageref{app:instance_distributions}
  \item[\ref{app:detailed_runtime}] Detailed Runtime Results \dotfill \pageref{app:detailed_runtime}
\end{itemize}

\clearpage
\section{Benchmarks and Details}
\label{app:benchmarks}

Three benchmarks share the same uncertainty model.
Every random coefficient is perturbed multiplicatively, $a(\boldsymbol{\xi}) = \bar{{a}} \odot \max(1 + \sigma\xi, 0)$ with $\xi \sim \mathcal{N}(0,1)$, where the truncation keeps the coefficient non-negative.
Perturbations of this form are common in the CCSP literature~\citep{luedtke2008sample,ahmed2008solving,nemirovski2006convex}.
We use $\sigma = 0.2$ on every constraint unless stated otherwise, and we evaluate methods on each benchmark at three sizes and three confidence levels $\alpha \in \{0.90, 0.95, 0.99\}$ over 20 held-out instances per setting, which gives 540 instances under Gaussian $\boldsymbol{\xi}$.
Two further settings reuse the knapsack generator and change only the uncertainty, (1) the heterogeneous $\sigma_d$ of Appendix~\ref{app:het_ablation} and (2) the four non-Gaussian distributions of Appendix~\ref{app:nongaussian}.

\subsection{Binary Chance-Constrained Knapsack}
The multi-dimensional binary chance-constrained knapsack with $n$ items and $d$ resource dimensions is given as
\begin{subequations}
 \begin{align}
  \max_{x \in \{0,1\}^n} \quad & \sum_{j=1}^n v_j x_j \\
  \text{s.t.} \quad &
  \text{Pr}_{\mathbf{\xi}}\!\left(\sum_{j=1}^n \tilde{w}_{i,j}(\boldsymbol{\xi})\, x_j \le b_i,\;
  \forall\, i=1,\dots,d\right) \ge \alpha.
\end{align}  \label{eq:kp}
\end{subequations}
where $x_j$ selects item $j$, $v_j > 0$ is its value, $w_{i,j} > 0$ its nominal weight in dimension $i$, and $b_i$ the deterministic capacity of dimension $i$.
A scenario carries one $\xi_j$ per item, so the perturbed weight is $\tilde{w}_{i,j}(\boldsymbol{\xi}) = w_{i,j} \max(1 + \sigma \xi_j,\, 0)$ with $\boldsymbol{\xi} \sim \mathcal{N}(0, I_n)$.

Weights are drawn as $w_{i,j} \sim \mathrm{Uniform}\{1,\dots,100\}$ and capacities as $b_i = \rho \sum_j w_{i,j}$ with $\rho = 0.5$.
Values follow the weakly correlated type of \citet{pisinger2005hard}, $v_j = \max(w_{1,j} + \epsilon,\, 1)$ with $\epsilon \sim \mathrm{Uniform}\{-10,\dots,10\}$.
The three sizes are $n{=}20$, $d{=}3$; $n{=}50$, $d{=}5$; and $n{=}100$, $d{=}10$.

\subsection{Maximum Weight Bipartite Matching}
\label{app:mwm}

The chance-constrained maximum weight bipartite matching with $d$ resource budgets is given as
\begin{subequations}
\begin{align}
  \max_{x\in\{0,1\}^{mn}} \quad & \sum_{i,j} p_{ij}\,x_{ij} \\
  \text{s.t.} \quad
  & \text{Pr}_{\mathbf{\xi}}\!\left(\sum_{i,j} \tilde{r}_{k,ij}(\boldsymbol{\xi})\,x_{ij} \le B_k,\;\forall\,k=1,\ldots,d\right) \ge \alpha, \\
  & \sum_{j} x_{ij} \le 1 \quad \forall\,i\in U, \qquad \sum_{i} x_{ij} \le 1 \quad \forall\,j\in V,
\end{align}
\end{subequations}
where $x_{ij}$ selects edge $(i,j)$ of a complete bipartite graph $G=(U,V,E)$ with $|U|=m$ and $|V|=n$, $p_{ij}$ is the edge profit, $r_{k,ij}$ its nominal consumption of resource $k$, and $B_k$ the budget of resource $k$.
A scenario carries one $\xi_{ij}$ per edge, so $\tilde{r}_{k,ij}(\boldsymbol{\xi})=r_{k,ij}\max(1+\sigma\xi_{ij},0)$ with $\xi_{ij}\sim\mathcal{N}(0,1)$.

Resources are drawn as $r_{k,ij}\sim\mathrm{Uniform}\{1,\dots,10\}$ and profits as $p_{ij}=\max(\bar{r}_{ij}+\epsilon,\,1)$ with $\epsilon\sim\mathrm{Uniform}\{-2,\dots,2\}$, where $\bar{r}_{ij}=d^{-1}\sum_k r_{k,ij}$.
Budgets are $B_k=\rho\min(m,n)\bar{r}_k$ with $\rho=0.5$, where $\bar{r}_k$ is the mean nominal consumption of resource $k$ and $\min(m,n)$ is the largest number of edges in a matching.
The three sizes are $m{=}n{=}10$, $15$ and $20$ at $d{=}5$, giving $100$, $225$ and $400$ variables.

\subsection{Stochastic Resource-Constrained Shortest Path}
\label{app:srcsp}

The resource-constrained shortest path~\citep{irnich2005shortest} with stochastic arc consumptions is given as
\begin{subequations}
 \begin{align}
  \max_{x\in\{0,1\}^{|A|}} \quad & \sum_{e\in A} c_e\,x_e \\
  \text{s.t.} \quad
  & \text{Pr}_{\mathbf{\xi}}\!\left(\sum_{e\in A} \tilde{r}_{i,e}(\boldsymbol{\xi})\,x_e \le B_i,\;\forall\,i=1,\dots,d\right) \ge \alpha, \label{eq:srcsp_cc}\\
  & \sum_{e:\,\mathrm{tail}(e)=v} x_e - \sum_{e:\,\mathrm{head}(e)=v} x_e = b_v \quad \forall\,v\in V, \label{eq:srcsp_flow}
\end{align}
\end{subequations}
where $x_e$ selects arc $e$ of a directed acyclic graph $G=(V,A)$, $c_e>0$ is the arc profit, $r_{i,e}>0$ its nominal consumption of resource $i$, and $B_i$ the budget of resource $i$.
The flow-conservation right-hand side is $b_s=1$, $b_t=-1$, and $b_v=0$ at every internal node.
A scenario carries one $\xi_e$ per arc, so $\tilde{r}_{i,e}(\boldsymbol{\xi}) = r_{i,e}\max(1+\sigma\xi_e,0)$ with $\xi_e \overset{\mathrm{iid}}{\sim} \mathcal{N}(0,1)$.

Each graph is a sequential backbone $0\!\to\!1\!\to\!\cdots\!\to\!(|V|-1)$, which guarantees an $s$-$t$ path, plus random skip arcs $(i,j)$ with $j \ge i+2$ up to exactly $|A|$ arcs.
Arc resources are drawn as $r_{i,e}\sim\mathrm{Uniform}\{1,\dots,100\}$ and profits as $c_e=\max(r_{1,e}+\epsilon,\,1)$ with $\epsilon\sim\mathrm{Uniform}\{-10,\dots,10\}$.
Budgets are $B_i = \rho\,(|V|-1)\,\bar{r}_i$ with $\rho=0.65$, where $\bar{r}_i$ is the mean nominal resource in dimension $i$ and $|V|-1$ is the largest number of arcs on a path.
The three sizes are $|V|{=}10$, $|A|{=}20$, $d{=}3$; $|V|{=}14$, $|A|{=}50$, $d{=}5$; and $|V|{=}18$, $|A|{=}100$, $d{=}10$.

\clearpage
\section{Baselines}
\label{app:baselines}

Each baseline solves one deterministic surrogate of~\eqref{eq:ccsp} and returns a single solution.
SAA and CVaR embed $|\mathcal{S}|$ i.i.d.\ scenarios $\mathcal{S} = \{\boldsymbol{\xi}^{(s)}\}_{s=1}^{|\mathcal{S}|}$ drawn from $\mathbb{P}$ through a set independent of the test set $\mathcal{T}$, and the Gaussian--Bonferroni (GB) approximation uses the model $\sigma$ without requiring scenarios.
All three are solved with Gurobi at 6 threads and a 3600\,s limit.
Each solution is then evaluated on $\mathcal{T}$ exactly as the candidate \method selects is, and counts as $\alpha$-feasible only if its empirical feasibility on $\mathcal{T}$ is at least $\alpha$.

\subsection{SAA Formulation}
\label{app:saa_formulation}

The sample average approximation (SAA) replaces the chance constraint of~\eqref{eq:ccsp} by its empirical counterpart over $\mathcal{S}$ and linearises it with one indicator per scenario~\citep{luedtke2008sample}:
\begin{equation}
\small
 \max_{\mathbf{x} \in \mathcal{X}} \Big\{f(\mathbf{x}) : \mathbf{A}(\boldsymbol{\xi}^{(s)})\mathbf{x} \le \mathbf{b}(\boldsymbol{\xi}^{(s)}) + \mathbf{M} (1 - y_s) \;\; \forall s = 1,\ldots, |\mathcal{S}|, \;\; \sum_{s = 1}^{|\mathcal{S}|} y_s \ge \lceil \alpha|\mathcal{S}| \rceil, \;\; \mathbf{y} \in \{0,1\}^{|\mathcal{S}|} \Big\},
\label{eq:saa}
\end{equation}
where $y_s = 1$ marks a scenario whose constraints are enforced and $\mathbf{M} \in \mathbb{R}^d_{+}$ is a vector of big-M constants, one per row.
The cardinality constraint allows at most $\lfloor (1-\alpha)|\mathcal{S}| \rfloor$ sampled scenarios to be violated, so~\eqref{eq:saa} is the chance constraint at risk level $1-\alpha$ under the empirical distribution of $\mathcal{S}$.
We use the plain big-M model without the strengthening inequalities of \citet{luedtke2010integer}. 
{We set $\mathbf{M}$ from the sampled scenarios with
$\hat a_{ij} = \max_{s} a_{ij}(\boldsymbol{\xi}^{(s)})$ the largest sampled coefficient and
$\bar b_i$ the right-hand side, which is deterministic in all three families,
\begin{equation*}
M_i =
\begin{cases}
\sum_{j=1}^{n} \hat a_{ij} - \bar b_i
  & \text{knapsack, where any subset of items may be selected,} \\[3pt]
\sum_{j \in T_i} \hat a_{ij} - \bar b_i
  & \text{matching, with $T_i$ the $\min(|U|,|V|)$ largest $\hat a_{ij}$,} \\[3pt]
(|V| - 1)\,\max_{j} \hat a_{ij} - \bar b_i
  & \text{shortest path, where a path uses at most $|V|-1$ arcs,}
\end{cases}
\end{equation*}
each floored at $10^{-3}$, so $M_i$ dominates the largest violation any $\mathbf{x} \in \mathcal{X}$
can have on $\mathcal{S}$.}

\subsection{Safe Approximations}
\label{app:safe_approx}

Safe approximations replace the chance constraint~\eqref{eq:cc} by a deterministic convex constraint that implies it, so every feasible point of the surrogate is $\alpha$-feasible by construction~\citep{nemirovski2006convex,chen2007robust,nemirovski2012safe}.
For a single row whose usage is a linear combination of independent zero-mean perturbations, the classical members of this family share the mean-plus-margin form of~\eqref{eq:gb_app} and differ only in the constant $\kappa$ and in what they assume about $\boldsymbol{\xi}$.
Each bounds an individual constraint, so a joint constraint over $d$ rows is handled either by splitting the risk across rows or by a joint distributionally robust reformulation~\citep{chen2010cvar,zymler2013distributionally}, whereas SAA enforces the joint constraint on the sampled scenarios directly.
We run the two members that apply to our benchmarks, the Gaussian reformulation with a Bonferroni split and the sample CVaR.

\paragraph{Bonferroni approximation.}
For a joint chance constraint over $d$ rows, the Bonferroni inequality requires each row $i$ to hold with probability at least $\alpha_i$ with $\sum_{i=1}^d (1 - \alpha_i) \leq 1 - \alpha$, and we use the uniform allocation $\alpha_i = 1 - (1-\alpha)/d$.
The bound ignores dependence between rows.
In our benchmarks every row is driven by the same item-level noise $\boldsymbol{\xi}$, so the rows are strongly positively dependent and the bound is loose.

\paragraph{Gaussian--Bonferroni (GB) approximation.}
The GB baseline imposes the joint constraint~\eqref{eq:cc} one row at a time.
Here we take a single row at level $\alpha$,
$\Pr(\mathbf{a}_i(\boldsymbol{\xi})^\top \mathbf{x} \leq \bar b_i) \geq \alpha$.
For the multiplicative Gaussian model, which a subset of our benchmarks uses,
$\mathbf{a}_i(\boldsymbol{\xi}) = \bar{\mathbf{a}}_i \odot (1 + \sigma_i \boldsymbol{\xi})$ with
$\boldsymbol{\xi} \sim \mathcal{N}(0, \mathbf{I})$, we have
$\mathbf{a}_i(\boldsymbol{\xi})^\top \mathbf{x} \sim \mathcal{N}\big(\bar{\mathbf{a}}_i^\top \mathbf{x},\,
\sigma_i^2 \|\bar{\mathbf{a}}_i \odot \mathbf{x}\|_2^2\big)$, so for $\alpha \geq 1/2$ this is
equivalent to the second-order cone constraint
\begin{equation}
\bar{\mathbf{a}}_i^\top \mathbf{x} + \Phi^{-1}(\alpha)\, \sigma_i \|\bar{\mathbf{a}}_i \odot \mathbf{x}\|_2 \leq \bar b_i,
\label{eq:gb_app}
\end{equation}
where $\Phi^{-1}(\cdot)$ is the inverse standard normal CDF~\citep[Section~2.6]{lobo1998applications}.
This is the exact form of the individual constraint, so GB is the tightest member of the family under the Gaussian model.
Imposing~\eqref{eq:gb_app} on every row with $\alpha$ replaced by the Bonferroni level
$\alpha_i = 1 - (1-\alpha)/d$ then guarantees~\eqref{eq:cc}.
When $\mathbf{x}$ is binary, the result is a mixed-integer second-order cone program.
The benchmarks clip the coefficients at zero, so~\eqref{eq:gb_app} is exact for the untruncated model only.  However, we note that with $\sigma_i = 0.2$, truncation among sampled scenarios is exceedingly rare.  

\paragraph{CVaR approximation.}
Conditional value-at-risk gives the tightest convex conservative approximation of an individual chance constraint~\citep{nemirovski2006convex,rockafellar2000optimization}: for the violation $g_i(\mathbf{x}, \boldsymbol{\xi}) = \mathbf{A}(\boldsymbol{\xi})_i \mathbf{x} - b_i(\boldsymbol{\xi})$ of row $i$ and a budget $\varepsilon_i \in (0,1)$,
\begin{equation}
\mathrm{CVaR}_{\varepsilon_i}\big(g_i(\mathbf{x}, \boldsymbol{\xi})\big) = \min_{t \in \mathbb{R}} \Big\{ t + \frac{1}{\varepsilon_i}\, \mathbb{E}\big[(g_i(\mathbf{x}, \boldsymbol{\xi}) - t)^+\big] \Big\} \leq 0
\label{eq:cvar_app}
\end{equation}
implies $\Pr(g_i(\mathbf{x}, \boldsymbol{\xi}) \leq 0) \geq 1 - \varepsilon_i$, the CVaR being the mean of the worst $\varepsilon_i$ fraction of outcomes.
We impose~\eqref{eq:cvar_app} row by row with the Bonferroni budgets $\varepsilon_i = (1-\alpha)/d$ and replace the expectation by its average over $\mathcal{S}$.
With the auxiliary variables of \citet{rockafellar2000optimization} the constraints are linear,
\begin{equation}
t_i + \frac{1}{\varepsilon_i |\mathcal{S}|} \sum_{s=1}^{|\mathcal{S}|} w_{s,i} \leq 0, \qquad
w_{s,i} \geq \mathbf{A}(\boldsymbol{\xi}^{(s)})_i \mathbf{x} - b_i(\boldsymbol{\xi}^{(s)}) - t_i, \qquad w_{s,i} \geq 0, \qquad i = 1, \ldots, d,
\label{eq:cvar_saa}
\end{equation}
adding $d$ thresholds and $d\,|\mathcal{S}|$ continuous variables to the MILP over $\mathbf{x}$, and no binaries.
The sample CVaR is conservative only once $\varepsilon_i|\mathcal{S}|$ is large, since when $\varepsilon_i|\mathcal{S}| \leq 1$ it is the largest sampled violation of the row and carries no guarantee.
CVaR needs no distributional assumption, so it is the safe approximation we run on the non-Gaussian distributions.

\clearpage
\section{Certified Feasibility Guarantees}
\label{app:certified_guarantees}

\methoddelta selects a candidate whose true feasibility admits a high-probability lower bound.  This section states the bound and the assumptions it needs.

\subsection{Setup}
\label{app:cert_setup}

Let $N$ candidate solutions $\hat{\mathbf{x}}_1,\dots,\hat{\mathbf{x}}_N$ be generated by \method \emph{without} using the evaluation set $\mathcal{V}$. The candidates are then compared on a single evaluation set $\mathcal{V} = \{\xi^{(1)},\dots,\xi^{(|\mathcal{V}|)}\}$ of $|\mathcal{V}|$ i.i.d.\ draws from the true (unknown) distribution $\mathbb{P}$, and the best feasible-looking one is selected. In our experiments, $N=1000$ candidates and $|\mathcal{V}|=10^6$ scenarios.
Selecting on $\mathcal{V}$ biases the winner, since the candidate that looks best is the one most likely to have drawn favourably, so the empirical rate needs a multiple-comparisons correction over the $N$ candidates.

For a candidate $\hat{\mathbf{x}}_j$, let
\[
p_j := \Pr_\xi\big(\mathbf{A}(\xi)\hat{\mathbf{x}}_j \le \mathbf{b}(\xi)\big)
\]
be its true (unknown) feasibility probability, and
\[
\hat{p}_j := \frac{1}{|\mathcal{V}|}\sum_{i=1}^{|\mathcal{V}|} \mathbb{I}\big\{\mathbf{A}(\xi^{(i)})\hat{\mathbf{x}}_j \le \mathbf{b}(\xi^{(i)})\big\}
\]
its empirical rate on $\mathcal{V}$.

\subsection{Proof of Proposition~\ref{prop:feasibility}}
\label{app:hoeffding}

We restate the proposition of Section~\ref{sec:guarantees} for reference.

\begin{propositionrestate}
Fix $\delta \in (0,1)$ and let $\hat{\mathbf{x}}_1, \ldots, \hat{\mathbf{x}}_N$ be generated
independently of $\mathcal{V}$. For each $j$, let $p_j = \Pr_{\boldsymbol{\xi}}(\mathbf{A}(\boldsymbol{\xi})\hat{\mathbf{x}}_j \leq \mathbf{b}(\boldsymbol{\xi}))$
be the true feasibility probability of $\hat{\mathbf{x}}_j$ and $\hat{p}_j$ its empirical rate on
$\mathcal{V}$. With probability at least $1 - \delta$ over the draw of $\mathcal{V}$,
\[
p_j \geq \hat{p}_j - \sqrt{\frac{\ln(N/\delta)}{2|\mathcal{V}|}} \quad \text{simultaneously for every } j = 1, \ldots, N.
\]
\end{propositionrestate}

\begin{proof}
The candidate solutions $\hat{\mathbf{x}}_1,\dots,\hat{\mathbf{x}}_N$ are generated without using $\mathcal{V}$. Thus, for any realized candidate pool and each candidate $j$, the $|\mathcal{V}|$ feasibility indicators evaluated on $\mathcal{V}$ are i.i.d.\ Bernoulli($p_j$).

The proposition's bound $p_j \geq \hat{p}_j - t$ fails exactly when $\hat{p}_j - p_j > t$, i.e., when the empirical rate overshoots the true rate by more than $t$. By Hoeffding's inequality~\citep{hoeffding1963}, for any $t > 0$:
\[
\Pr\big(\hat{p}_j - p_j \geq t\big) \leq e^{-2|\mathcal{V}|t^2}.
\]
Taking $t = \sqrt{\ln(N/\delta)/(2|\mathcal{V}|)}$ makes the right-hand side equal $\delta/N$, so this failure event occurs with probability at most $\delta/N$.

By Boole's inequality (union bound), the probability that the bound fails for \emph{any} candidate is:
\[
\Pr\big(\exists j: \text{bound fails for } j\big) \leq \sum_{j=1}^N \Pr\big(\text{bound fails for } j\big) \leq \sum_{j=1}^N \frac{\delta}{N} = \delta.
\]
So the bound holds for all $N$ candidates simultaneously with probability at least $1-\delta$.  The event does not reference which $j$ is selected, so it covers whichever index the selection rule returns.
\end{proof}

\textbf{Certified selection rule.} To certify that the selected solution has true feasibility $p_j \geq \alpha$ with probability at least $1-\delta$, we select only candidates satisfying:
\[
\hat{p}_j \geq \alpha + \epsilon, \quad \text{where } \epsilon = \sqrt{\frac{\ln(N/\delta)}{2|\mathcal{V}|}}.
\]

\subsection{Relation to SAA Theory}

Proposition~\ref{prop:feasibility} is a finite-candidate specialization of the feasibility argument in \citet{luedtke2008sample}.
Their union bound runs over every decision in the sample-approximation feasible region that is truly infeasible but might look empirically feasible, which gives a multiplicity factor based on the size of the feasible region, which in general is exponential for a bounded integer program. They note that the resulting sample sizes are far too large to solve.
\method generates its candidates without $\mathcal{V}$, so only those $N$ can be selected and the correction runs over that finite set instead.
It needs i.i.d.\ draws for $\mathcal{V}$ and a pool generated without it, and no distributional assumptions.

\subsection{Practical Implications}

For our experiments with $|\mathcal{V}| = 10^6$, $N = 1000$, and $\delta = 0.05$, the required margin is:
\[
\epsilon = \sqrt{\frac{\ln(1000/0.05)}{2 \times 10^6}} = \sqrt{\frac{\ln(20000)}{2 \times 10^6}} \approx 0.0022 \quad (0.22\%).
\]

Raising the threshold by $0.22\%$ has a marginal impact on the objective so the certificate is close to free.
The margin scales as $\epsilon \propto 1/\sqrt{|\mathcal{V}|}$, and feasibility evaluation is $2$--$13\%$ of \method runtime (Appendix~\ref{app:detailed_runtime}), so a tenfold larger $\mathcal{V}$ tightens the bound by $\sqrt{10} \approx 3.2$ at a cost that stays small.
Table~\ref{tab:hoeffding_comparison} compares the two bounds at the same empirical rate.
Even granting SAA a multiplicity factor of one, which ignores the selection its own optimization performs, its bound sits $3.9\%$ below the empirical rate against $0.22\%$ for ours.

\begin{table}[h]
\centering
\caption{Hoeffding lower bounds on true feasibility for an observed empirical rate of $\hat{p}=0.99$, with $\delta=0.05$.}
\label{tab:hoeffding_comparison}
\begin{tabular}{lccc}
\toprule
Method & $|\mathcal{V}|$ & Empirical $\hat{p}$ & Hoeffding lower bound \\
\midrule
SAA-favorable ($S=1000$, $N=1$) & 1,000 & 0.99 & 0.951 \\
\method ($|\mathcal{V}|=10^6$, $N=1000$) & 1,000,000 & 0.99 & 0.988 \\
\bottomrule
\end{tabular}
\end{table}

\textbf{Note on guarantee scope.} This bounds feasibility, not optimality.  We do not claim the returned solution is close to the true chance-constrained optimum, only that its feasibility probability admits the stated lower bound.

\clearpage
\section{Algorithmic Details}
\label{app:boccp_analysis}

This section gives the hyperparameters and the empirical behaviour of the search over $(\gamma, \beta)$.

\subsection{Hyperparameter Settings}
\label{app:hyperparameters}

Table~\ref{tab:hyperparams} summarizes the hyperparameters used for \method in all experiments. The total number of MILP solves per instance is $N_{\text{scen}} \times (N_{\text{init}} + N_{\text{iter}}) + 40 = 1000 \times 15 + 40 = 15{,}040$, of which the selection in line 10 of Algorithm~\ref{alg:access} retains one per scenario, giving $N = 1000$ candidates for selection.  Each solve is single-threaded with a 3\,s limit, and scenarios are processed in parallel.

\begin{table}[h!]
  \centering
  \caption{\method hyperparameter settings used in all experiments.}
  \label{tab:hyperparams}
  \small
  \begin{tabular}{lrl}
    \toprule
    Parameter & Value & Description \\
    \midrule
    \multicolumn{3}{l}{\textit{Scenario Sampling}} \\
    $N_{\text{scen}}$ & 1000 & Scenarios drawn from the problem's own distribution $\mathbb{P}$ \\
    \midrule
    \multicolumn{3}{l}{\textit{Bayesian Optimization (per scenario)}} \\
    $N_{\text{init}}$ & 5 & Initial Sobol samples \\
    $N_{\text{iter}}$ & 10 & BO iterations after initialization \\
    Acquisition & qLogEI & Log expected improvement, feasibility-constrained \\
    \midrule
    \multicolumn{3}{l}{\textit{Search Space}} \\
    $\gamma$ range & $[0,\, 2\hat{\gamma}_0]$ & Tightening; $\hat{\gamma}_0$ calibrated per instance, capped at 20 \\
    $\beta$ range & $[0,\, 1]$ & Mixing weight, nominal to sampled scenario \\
    Bracket points & 40 & Nominal solves calibrating $\hat{\gamma}_0$ \\
    \midrule
    \multicolumn{3}{l}{\textit{Subproblem}} \\
    Time limit & 3\,s & Per MILP solve, single thread \\
    MILP gap & $10^{-4}$ & Gurobi default \\
    \midrule
    \multicolumn{3}{l}{\textit{Evaluation}} \\
    $N$ (candidates) & 1000 & Number of candidate solutions for final selection \\
    $|\mathcal{V}'|$  & $10^4$ & Scenarios scoring each point within a scenario's search \\
    $|\mathcal{V}|$   & $10^6$ & Scenarios selecting among the $N$ pooled candidates \\
    $|\mathcal{T}|$   & $10^6$ & Held-out scenarios for evaluation (disjoint from $\mathcal{V}$/$\mathcal{V}'$) \\
    \midrule
    \multicolumn{3}{l}{\textit{Certified Feasibility (optional)}} \\
    $\delta$ & 0.05 & Failure probability for the Hoeffding bound \\
    \bottomrule
  \end{tabular}
\end{table}

\subsection{Estimating the Per-Row Scale}
\label{app:sigma_estimate}

This section details how the per-row scale $\hat{\boldsymbol{\sigma}}$ is computed.
Appendix~\ref{app:het_ablation} reports an ablation on its efficacy when the uncertainty is heterogeneous across constraints.
Throughout, $\mathcal{S}_\sigma = \{\boldsymbol{\xi}^{(1)}, \ldots, \boldsymbol{\xi}^{(n_\sigma)}\}$ is a set of $n_\sigma$ scenarios drawn from $\mathbb{P}$ for this estimate alone, disjoint from the scenarios used to generate candidates and from $\mathcal{V}'$, $\mathcal{V}$ and $\mathcal{T}$.

\paragraph{Left-hand-side uncertainty.}
In the context of our experiments, we consider multiplicative left-hand-side uncertainty, and measure for each row how far its coefficients move relative to their nominal values,
\begin{equation}
r_{d,i}(\boldsymbol{\xi}) \;=\; \frac{\mathbf{A}(\boldsymbol{\xi})_{d,i} - \bar{\mathbf{A}}_{d,i}}{\lvert \bar{\mathbf{A}}_{d,i}\rvert},
\qquad
\hat{\sigma}_{d} \;=\; \frac{1}{n}\sum_{i=1}^{n}
\sqrt{\frac{1}{\lvert\mathcal{S}_\sigma\rvert - 1}\sum_{\boldsymbol{\xi} \in \mathcal{S}_\sigma}
\Big(r_{d,i}(\boldsymbol{\xi}) - \bar{r}_{d,i}\Big)^{2}},
\label{eq:sigma_estimate}
\end{equation}
where $i$ indexes the $n$ decision variables and $d$ the rows of the joint constraint, $\mathbf{A}(\boldsymbol{\xi})_{d,i}$ is the coefficient of variable $i$ in row $d$ under scenario $\boldsymbol{\xi}$, $\bar{\mathbf{A}}_{d,i}$ is that same coefficient in the nominal scenario, and $\bar{r}_{d,i}$ is the mean of $r_{d,i}$ over $\mathcal{S}_\sigma$.
Each row therefore carries the average relative volatility of its coefficients, measured in units of their nominal values.

\paragraph{Generalizing beyond left-hand-side uncertainty.}
The same computation applies to the right-hand side, with $r_{d,i}$ replaced by $(b(\boldsymbol{\xi})_d - \bar{b}_d)/\lvert\bar{b}_d\rvert$.
A formulation carrying uncertainty on one side only uses the corresponding scale directly.  However, when both are uncertain, we need to combine the two deviations to construct a single scale per row, e.g., by summing them.

\subsection{BO Convergence}
\label{app:boccp_convergence}

In this section, we analyze the performance as a function of the number of iterations of BO.
Within a scenario, the running best is the best objective among the points solved so far whose quick feasibility clears $\alpha$.
Figure~\ref{fig:bo_convergence} plots it after $k$ BO iterations as a percentage over the best of the initial quasi-random samples, averaged over all instance/scenario sets.  
Generally, BO increases with the number of iterations, with some exceptions, e.g., SRCSP at low $\alpha$, where the initialization was good enough to find high-quality solutions.  

\begin{figure}[h!]
    \centering
    \includegraphics[width=\textwidth]{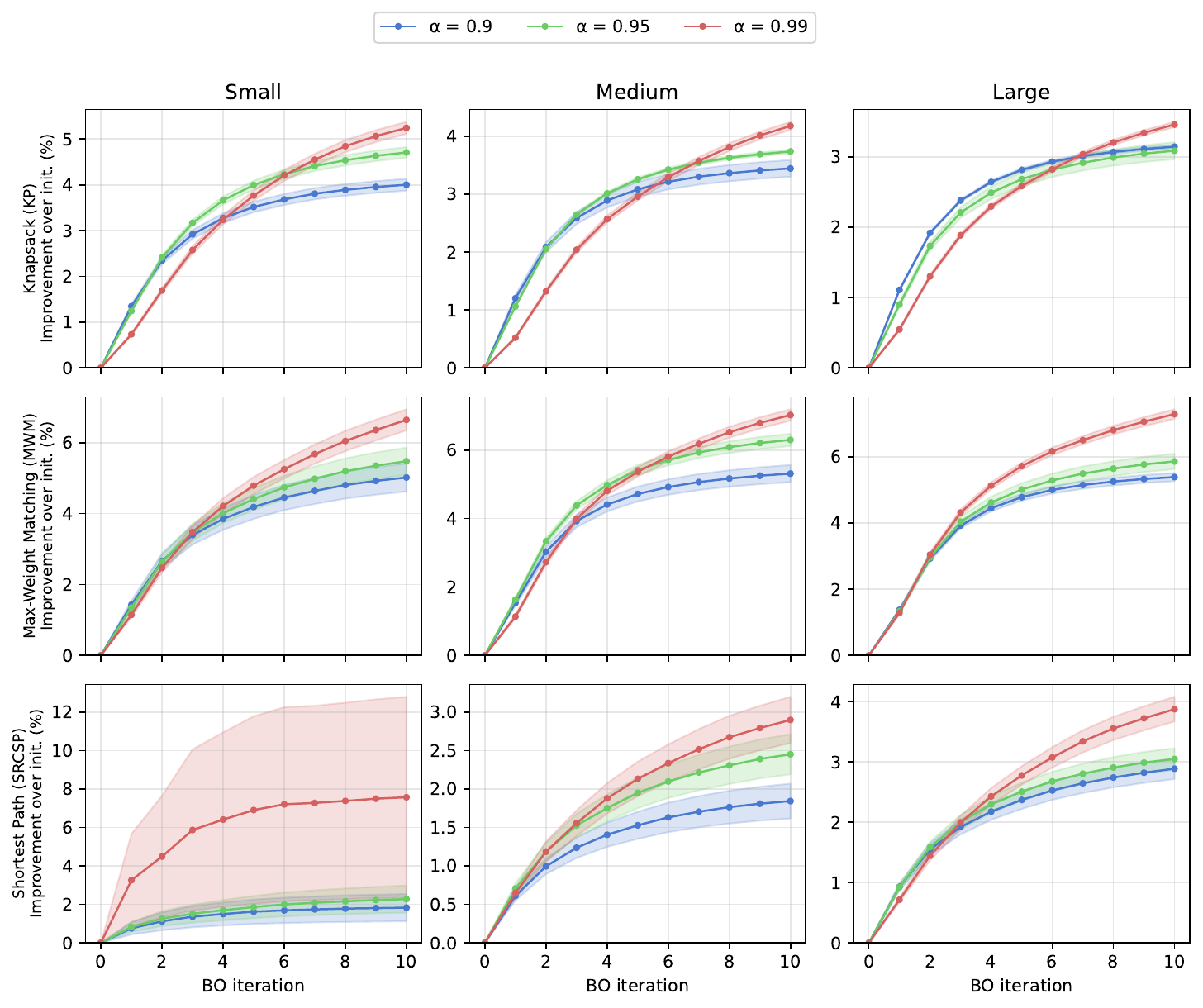}
    \caption{Improvements based on the number of BO iterations.  Rows are benchmarks and columns are instance sizes.  Iteration 0 is the end of the initial design, so the curve starts at 0\% by construction. Shaded regions are the standard error across the 20 instances.}
    \label{fig:bo_convergence}
\end{figure}

\subsection{Scaling with Candidate Budget}
\label{app:boccp_scaling}

Figure~\ref{fig:fo_gap_vs_n} shows how solution quality (measured by FO-gap) scales with the number of candidates $N$. For each instance, we subsample the first $N$ candidates from the full \method run and report the FO-gap relative to the best feasible objective at $N{=}1000$.
The curves show rapid initial improvement followed by diminishing returns.
SRCSP converges almost immediately, reaching the full-pool objective by $N{=}50$ in every cell.
KP and MWM need more candidates but close to convergence on average by $N{=}100$, although individual instances still benefit from more iterations.  

\begin{figure}[h!]
    \centering
    \includegraphics[width=\textwidth]{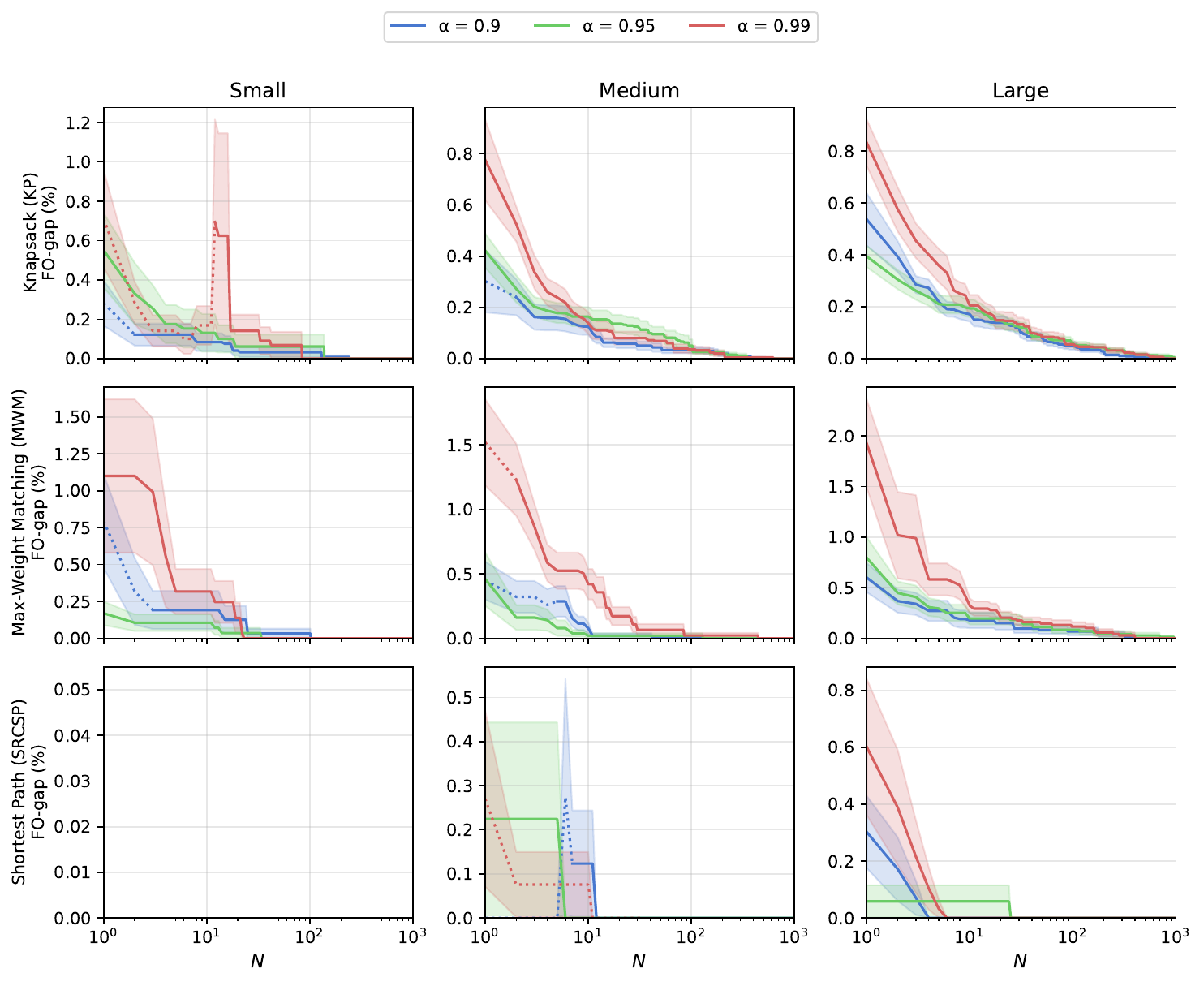}
    \caption{FO-gap vs.\ number of candidates $N$ (log scale). Rows correspond to benchmarks (KP, MWM, SRCSP) and columns to instance sizes (Small, Medium, Large). FO-gap is computed relative to the best FO at $N{=}1000$. Higher $\alpha$ (red) generally requires more candidates to achieve low FO-gap. Shaded regions show standard error across 20 test instances.}
    \label{fig:fo_gap_vs_n}
\end{figure}

\subsection{Parameter Space Analysis}
\label{app:boccp_param_analysis}

Figure~\ref{fig:optimal_params} shows the selected parameters $(\gamma^\star, \beta^\star)$ for each test instance, where $(\gamma^\star, \beta^\star)$ are the parameters of the candidate chosen by the certified rule. The scatter plots reveal distinct patterns across benchmarks.  For KP, the selected $\gamma^\star$ concentrates in $[0.3, 0.8]$ with $\beta^\star$ close to zero.  For MWM, $\gamma^\star$ is small with $\beta^\star$ near zero at $\alpha \in \{.90, .95\}$ and rising at $\alpha{=}.99$.  Lastly, for SRCSP both parameters spread more broadly, although this may be due to the instances being easier and having more viable good solutions.

\begin{figure}[h!]
    \centering
    \includegraphics[width=\textwidth]{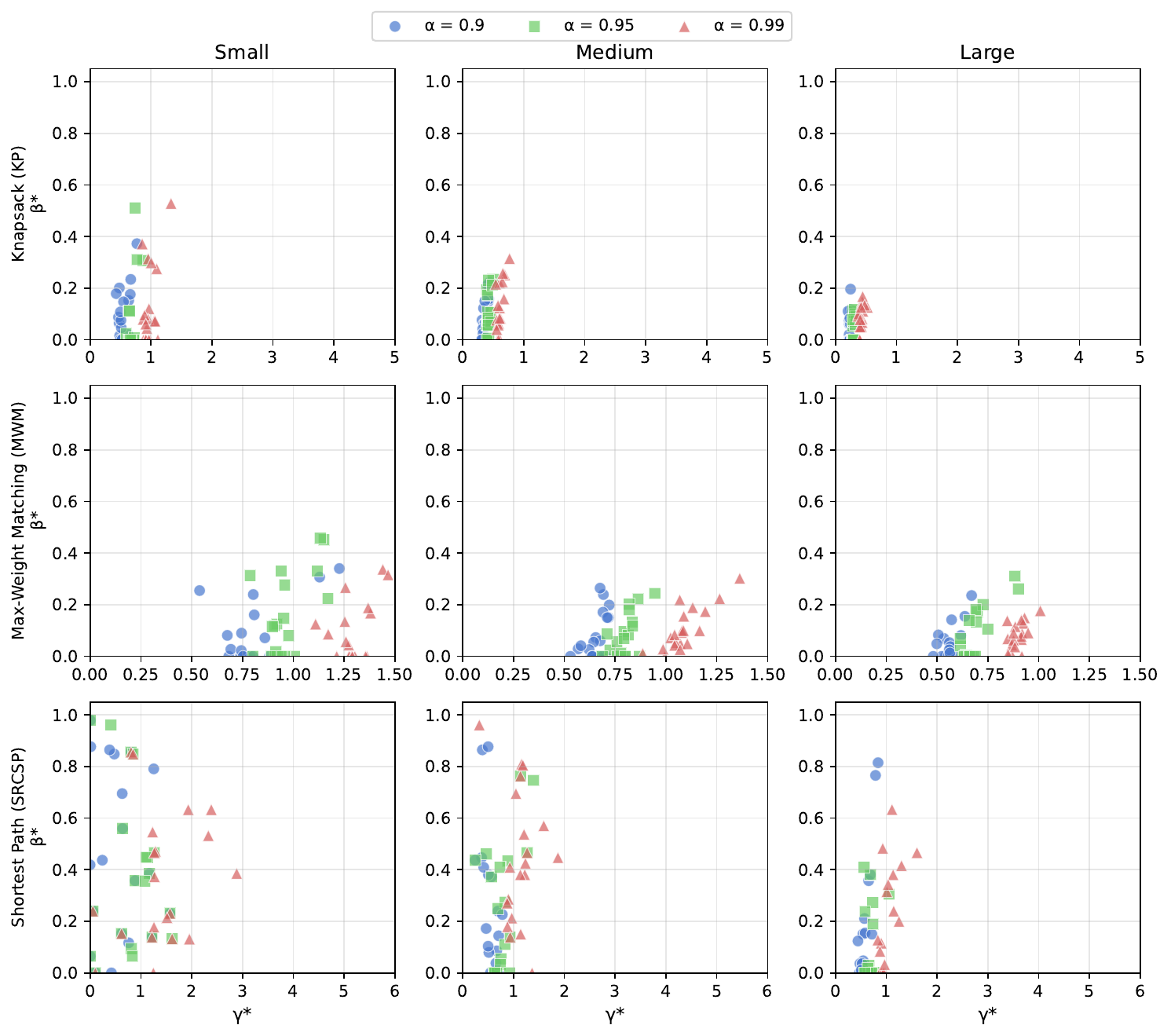}
    \caption{Selected parameters $(\gamma^\star, \beta^\star)$ per instance. Each point represents one test instance; colors denote $\alpha \in \{0.90, 0.95, 0.99\}$. Rows correspond to benchmarks and columns to instance sizes.}
    \label{fig:optimal_params}
\end{figure}

\clearpage
\section{Non-Standard Uncertainty}
\label{app:nonstandard}

The benchmarks of Appendix~\ref{app:benchmarks} use a single homogeneous $\sigma$ and
i.i.d.\ Gaussian $\xi_i$, which is what the classical safe approximations assume.
\method needs neither.  It draws scenarios from the problem's own distribution and scores candidates
by their empirical feasibility, so neither the search nor the selection rule refers to the
distribution.
We test this in two ways.  Appendix~\ref{app:het_ablation} varies the scale across constraints and
keeps $\boldsymbol{\xi}$ Gaussian, and Appendix~\ref{app:nongaussian} replaces the distribution
itself with four non-Gaussian alternatives.  The instances, seeds and multiplicative model are
those of Appendix~\ref{app:benchmarks} throughout, and only the uncertainty changes.

\subsection{Heterogeneous Uncertainty Ablation}
\label{app:het_ablation}

\method normalizes constraint slack by a per-row $\sigma_d$, so rows with different noise levels are tightened differently.
We compare against a variant that uses the single scalar $\bar{\sigma} = \frac{1}{d}\sum_{i=1}^d \sigma_i$ instead.
We run the baselines on the same instances, so the cost of the scalar simplification can be read against what they pay rather than against \method alone.

We evaluate on four KP variants at medium size ($n{=}50$, $d{=}5$), which differ only in
$\boldsymbol{\sigma}$ and share the same mean within each pair:
\begin{itemize}[nosep,leftmargin=*]
  \item \textbf{hom\_20}: $\boldsymbol{\sigma} = (0.20, 0.20, 0.20, 0.20, 0.20)$, i.e., the Medium KP of the main results for reference
  \item \textbf{hom\_35}: $\boldsymbol{\sigma} = (0.35, 0.35, 0.35, 0.35, 0.35)$
  \item \textbf{het\_20}: $\boldsymbol{\sigma} = (0.02, 0.11, 0.20, 0.29, 0.38)$
  \item \textbf{het\_35}: $\boldsymbol{\sigma} = (0.02, 0.185, 0.35, 0.515, 0.68)$
\end{itemize}
The heterogeneous rows are evenly spaced about the same mean as their homogeneous
counterpart, so a scalar $\bar{\sigma}$ is the correct mean in every setting and only its spread
changes.

Table~\ref{tab:het_ablation} reports the comparison.  \method performs similarly across the four settings as it consistently computes best-known feasible.
The scalar-$\bar{\sigma}$ variant matches it on the homogeneous settings, where the two normalizations
coincide, and degrades as the spread grows and as $\alpha$ tightens, since one $\bar{\sigma}$ misprices
the rows that bind.  At its worst it pays $10.23\%$ against GB's $0.68\%$ on the same instances.
The per-row $\sigma_d$ is therefore not a refinement, and justified to handle more general classes of uncertainty.

\begin{table}[h!]
  \caption{Heterogeneous uncertainty on the knapsack benchmark ($n{=}50$, $d{=}5$).
    Each cell reports FO-gap\,/\,$\alpha$-feas rate over 20 seeds, with feasibility measured on
    $10^6$ held-out scenarios.  The four settings share instances and seeds and differ only in the
    per-row $\sigma_d$; hom\_20 is Medium KP from the main results.}
  \label{tab:het_ablation}
  \centering
  \small
  \setlength\tabcolsep{4pt}
  \begin{tabular}{ll rrr}
    \toprule
    & Method & $\alpha{=}.90$ & $\alpha{=}.95$ & $\alpha{=}.99$ \\
    \midrule
    \multirow{8}{*}{\textbf{hom\_20}}
      & SAA (S=1k) & 0.21\%\,/\,40\% & 0.19\%\,/\,25\% & ---\,/\,0\% \\
      & SAA (S=10k) & 0.12\%\,/\,90\% & 0.04\%\,/\,75\% & 0.03\%\,/\,30\% \\
      & CVaR (S=1k) & 4.16\%\,/\,100\% & 3.40\%\,/\,100\% & 1.37\%\,/\,85\% \\
      & CVaR (S=10k) & 4.51\%\,/\,100\% & 4.02\%\,/\,100\% & 3.12\%\,/\,100\% \\
      & GB & 3.00\%\,/\,100\% & 2.65\%\,/\,100\% & 2.14\%\,/\,100\% \\
      & \method (scalar $\bar{\sigma}$) & 0.00\%\,/\,95\% & 0.04\%\,/\,100\% & 0.00\%\,/\,100\% \\
      & \method & 0.00\%\,/\,95\% & 0.04\%\,/\,100\% & 0.00\%\,/\,100\% \\
      & \methoddelta & 0.06\%\,/\,100\% & 0.11\%\,/\,100\% & 0.39\%\,/\,100\% \\
    \midrule
    \multirow{8}{*}{\textbf{hom\_35}}
      & SAA (S=1k) & 0.37\%\,/\,15\% & 0.57\%\,/\,5\% & ---\,/\,0\% \\
      & SAA (S=10k) & 0.21\%\,/\,90\% & 0.28\%\,/\,70\% & 0.18\%\,/\,35\% \\
      & CVaR (S=1k) & 6.65\%\,/\,100\% & 5.08\%\,/\,100\% & 1.72\%\,/\,60\% \\
      & CVaR (S=10k) & 7.52\%\,/\,100\% & 6.58\%\,/\,100\% & 4.87\%\,/\,100\% \\
      & GB & 5.04\%\,/\,100\% & 4.38\%\,/\,100\% & 3.40\%\,/\,100\% \\
      & \method (scalar $\bar{\sigma}$) & 0.06\%\,/\,100\% & 0.03\%\,/\,100\% & 0.00\%\,/\,90\% \\
      & \method & 0.04\%\,/\,100\% & 0.03\%\,/\,100\% & 0.00\%\,/\,90\% \\
      & \methoddelta & 0.15\%\,/\,100\% & 0.17\%\,/\,100\% & 0.65\%\,/\,100\% \\
    \midrule
    \multirow{8}{*}{\textbf{het\_20}}
      & SAA (S=1k) & 0.02\%\,/\,90\% & 0.03\%\,/\,75\% & 0.01\%\,/\,40\% \\
      & SAA (S=10k) & 0.15\%\,/\,100\% & 0.07\%\,/\,90\% & 0.00\%\,/\,95\% \\
      & CVaR (S=1k) & 0.86\%\,/\,100\% & 0.72\%\,/\,100\% & 0.46\%\,/\,100\% \\
      & CVaR (S=10k) & 0.89\%\,/\,100\% & 0.76\%\,/\,100\% & 0.65\%\,/\,100\% \\
      & GB & 0.58\%\,/\,100\% & 0.46\%\,/\,100\% & 0.36\%\,/\,100\% \\
      & \method (scalar $\bar{\sigma}$) & 0.10\%\,/\,100\% & 0.43\%\,/\,100\% & 2.93\%\,/\,90\% \\
      & \method & 0.00\%\,/\,100\% & 0.00\%\,/\,95\% & 0.00\%\,/\,100\% \\
      & \methoddelta & 0.02\%\,/\,100\% & 0.02\%\,/\,100\% & 0.07\%\,/\,100\% \\
    \midrule
    \multirow{8}{*}{\textbf{het\_35}}
      & SAA (S=1k) & 0.11\%\,/\,60\% & 0.05\%\,/\,60\% & 0.03\%\,/\,25\% \\
      & SAA (S=10k) & 0.44\%\,/\,95\% & 0.24\%\,/\,95\% & 0.00\%\,/\,50\% \\
      & CVaR (S=1k) & 1.31\%\,/\,100\% & 1.15\%\,/\,100\% & 0.75\%\,/\,100\% \\
      & CVaR (S=10k) & 1.36\%\,/\,100\% & 1.32\%\,/\,100\% & 1.10\%\,/\,100\% \\
      & GB & 0.78\%\,/\,100\% & 0.73\%\,/\,100\% & 0.68\%\,/\,100\% \\
      & \method (scalar $\bar{\sigma}$) & 2.00\%\,/\,100\% & 4.63\%\,/\,95\% & 10.23\%\,/\,100\% \\
      & \method & 0.00\%\,/\,100\% & 0.00\%\,/\,100\% & 0.02\%\,/\,95\% \\
      & \methoddelta & 0.01\%\,/\,100\% & 0.06\%\,/\,100\% & 0.17\%\,/\,100\% \\
    \bottomrule
  \end{tabular}
\end{table}

\subsection{Non-Gaussian Uncertainty}
\label{app:nongaussian}

The mixture of \eqref{eq:rescaling} makes no distributional assumption: it requires only that
constraint data can be mixed, not that $\boldsymbol{\xi}$ be Gaussian, independent or light-tailed.
To test this we repeat the knapsack benchmark under four alternative distributions.
The instances, the seeds, and the multiplicative model
$\mathbf{A}(\boldsymbol{\xi})_{d,i} = w_{d,i}\max(1 + \sigma_d \xi_i,\, 0)$ are identical to the
Gaussian benchmark.  Only the distribution of $\boldsymbol{\xi} \in \mathbb{R}^n$ changes, and each
one below is zero-mean.

\begin{itemize}[nosep,leftmargin=*]
  \item \textbf{Student-$t_3$}, heavy tails with independent coordinates:
        \[
          \xi_i = t_i \big/ \sqrt{\nu/(\nu-2)}, \qquad t_i \overset{\text{iid}}{\sim} t_\nu, \quad \nu = 3 .
        \]

  \item \textbf{Jump}, rare disruptions with a discrete right-skewed component:
        \[
          \xi_i = z_i + J\,(B_i - p), \qquad z_i \overset{\text{iid}}{\sim} \mathcal{N}(0,1), \quad
          B_i \overset{\text{iid}}{\sim} \operatorname{Bernoulli}(p), \quad p = 0.03,\; J = 6 ,
        \]
        with $z$ and $B$ independent.  A coefficient jumps by roughly $\sigma_d J$ with probability
        $p$.
  \item \textbf{Clustered}, correlation aligned with the objective:
        \[
          \xi_i = \sqrt{\rho}\, f_{c(i)} + \sqrt{1-\rho}\, z_i, \qquad
          f_1, f_2, z_i \overset{\text{iid}}{\sim} \mathcal{N}(0,1), \quad \rho = 0.85 ,
        \]
        where $f_1$ and $f_2$ are one factor per cluster and $c(i) \in \{1,2\}$ splits the items by value-to-mean-weight ratio\textcolor{red}{,} and cluster $1$ holds
        the $n/2$ most attractive.  The alignment is deliberate: the items a solution wants to take are
        the ones whose uncertainty moves together, so diversifying costs objective value.

  \item \textbf{Risky subset}, item scales that differ within a row:
        \[
          \xi_i = \begin{cases}
            s\,\big(\sqrt{\rho}\, f + \sqrt{1-\rho}\, z_i\big), & i \in \mathcal{R}, \\[2pt]
            z_i, & \text{otherwise,}
          \end{cases}
          \qquad f, z_i \overset{\text{iid}}{\sim} \mathcal{N}(0,1), \quad s = 3, \; \rho = 0.8 ,
        \]
        where $\mathcal{R}$ is a fixed random $30\%$ of the items, drawn from the instance seed.  The
        risky items have standard deviation $s$ and share the single factor $f$; the rest keep the benchmark
        scale and stay independent.
\end{itemize}

Each distribution stresses a different assumption.  Student-$t_3$ removes light-tail structure, 
jump introduces a skewed discrete component with no Gaussian analogue,
clustering removes independence in the way that most penalises a per-row scalar tightening, and the
risky subset gives the items within a row scales that differ by a factor of three.
We omit the GB baseline: its per-row quantile
presumes i.i.d.\ standard
Gaussian $\xi_i$, and each distribution violates a different part of that.
SAA and CVaR are distribution-free in construction and draw from the same scenario generator, so they
transfer unchanged.

\begin{table}[h!]
  \caption{Non-Gaussian uncertainty on the knapsack benchmark.  Each cell reports
    FO-gap\,/\,$\alpha$-feas rate over 20 seeds, with feasibility measured on $10^6$ held-out
    scenarios from the same distribution.}
  \label{tab:distributions}
  \centering
  \resizebox{\textwidth}{!}{%
  \setlength\tabcolsep{3pt}
  \begin{tabular}{ll rrr rrr}
    \toprule
    & & \multicolumn{3}{c}{$n{=}50$, $d{=}5$} & \multicolumn{3}{c}{$n{=}100$, $d{=}10$} \\
    \cmidrule(lr){3-5}\cmidrule(lr){6-8}
    & Method & $\alpha{=}.90$ & $\alpha{=}.95$ & $\alpha{=}.99$
           & $\alpha{=}.90$ & $\alpha{=}.95$ & $\alpha{=}.99$ \\
    \midrule
    \multirow{6}{*}{\textbf{Student-$t_3$}}
      & SAA (S=1k) & 0.13\%\,/\,40\% & 0.15\%\,/\,25\% & ---\,/\,0\% & 0.00\%\,/\,10\% & ---\,/\,0\% & ---\,/\,0\% \\
      & SAA (S=10k) & 0.44\%\,/\,80\% & 0.12\%\,/\,55\% & 0.03\%\,/\,40\% & 0.35\%\,/\,95\% & 0.06\%\,/\,70\% & 0.02\%\,/\,25\% \\
      & CVaR (S=1k) & 5.24\%\,/\,100\% & 4.77\%\,/\,100\% & 2.84\%\,/\,50\% & 3.86\%\,/\,100\% & 3.09\%\,/\,100\% & 1.19\%\,/\,20\% \\
      & CVaR (S=10k) & 6.31\%\,/\,100\% & 6.71\%\,/\,100\% & 7.80\%\,/\,100\% & 5.23\%\,/\,100\% & 5.43\%\,/\,100\% & 4.88\%\,/\,100\% \\
      & \method & 0.00\%\,/\,100\% & 0.01\%\,/\,100\% & 0.00\%\,/\,90\% & 0.00\%\,/\,100\% & 0.03\%\,/\,95\% & 0.02\%\,/\,85\% \\
      & \methoddelta & 0.05\%\,/\,100\% & 0.13\%\,/\,100\% & 0.71\%\,/\,100\% & 0.02\%\,/\,100\% & 0.07\%\,/\,100\% & 0.37\%\,/\,100\% \\
    \midrule
    \multirow{6}{*}{\textbf{Jump}}
      & SAA (S=1k) & 0.32\%\,/\,15\% & 0.34\%\,/\,15\% & ---\,/\,0\% & 0.02\%\,/\,15\% & ---\,/\,0\% & ---\,/\,0\% \\
      & SAA (S=10k) & 0.04\%\,/\,55\% & 0.03\%\,/\,75\% & 0.00\%\,/\,20\% & 0.17\%\,/\,85\% & 0.04\%\,/\,55\% & 0.00\%\,/\,25\% \\
      & CVaR (S=1k) & 7.44\%\,/\,100\% & 6.00\%\,/\,100\% & 1.81\%\,/\,50\% & 4.92\%\,/\,100\% & 3.49\%\,/\,100\% & ---\,/\,0\% \\
      & CVaR (S=10k) & 8.54\%\,/\,100\% & 7.85\%\,/\,100\% & 5.86\%\,/\,100\% & 6.50\%\,/\,100\% & 5.78\%\,/\,100\% & 3.77\%\,/\,100\% \\
      & \method & 0.06\%\,/\,95\% & 0.10\%\,/\,95\% & 0.05\%\,/\,95\% & 0.10\%\,/\,95\% & 0.13\%\,/\,95\% & 0.08\%\,/\,100\% \\
      & \methoddelta & 0.14\%\,/\,100\% & 0.24\%\,/\,100\% & 0.93\%\,/\,100\% & 0.13\%\,/\,100\% & 0.20\%\,/\,100\% & 0.59\%\,/\,100\% \\
    \midrule
    \multirow{6}{*}{\textbf{Clustered}}
      & SAA (S=1k) & 0.48\%\,/\,50\% & 0.52\%\,/\,35\% & 0.71\%\,/\,30\% & 0.58\%\,/\,30\% & 0.61\%\,/\,25\% & 0.56\%\,/\,15\% \\
      & SAA (S=10k) & 0.07\%\,/\,70\% & 0.07\%\,/\,80\% & 0.15\%\,/\,40\% & 0.09\%\,/\,70\% & 0.08\%\,/\,45\% & 0.06\%\,/\,55\% \\
      & CVaR (S=1k) & 10.93\%\,/\,100\% & 9.26\%\,/\,100\% & 5.24\%\,/\,100\% & 12.53\%\,/\,100\% & 10.35\%\,/\,100\% & 5.62\%\,/\,95\% \\
      & CVaR (S=10k) & 10.84\%\,/\,100\% & 9.36\%\,/\,100\% & 6.91\%\,/\,100\% & 12.64\%\,/\,100\% & 10.90\%\,/\,100\% & 8.22\%\,/\,100\% \\
      & \method & 0.14\%\,/\,100\% & 0.19\%\,/\,100\% & 0.07\%\,/\,90\% & 0.09\%\,/\,100\% & 0.10\%\,/\,95\% & 0.07\%\,/\,80\% \\
      & \methoddelta & 0.21\%\,/\,100\% & 0.31\%\,/\,100\% & 0.80\%\,/\,100\% & 0.17\%\,/\,100\% & 0.22\%\,/\,100\% & 1.03\%\,/\,100\% \\
    \midrule
    \multirow{6}{*}{\textbf{Risky subset}}
      & SAA (S=1k) & 0.12\%\,/\,60\% & 0.06\%\,/\,40\% & ---\,/\,0\% & 0.00\%\,/\,30\% & 0.00\%\,/\,20\% & 0.00\%\,/\,5\% \\
      & SAA (S=10k) & 1.21\%\,/\,85\% & 0.43\%\,/\,90\% & 0.09\%\,/\,80\% & 1.44\%\,/\,95\% & 0.32\%\,/\,75\% & 0.00\%\,/\,35\% \\
      & CVaR (S=1k) & 4.15\%\,/\,100\% & 3.40\%\,/\,100\% & 1.55\%\,/\,95\% & 3.12\%\,/\,100\% & 2.28\%\,/\,100\% & 0.69\%\,/\,75\% \\
      & CVaR (S=10k) & 4.42\%\,/\,100\% & 3.88\%\,/\,100\% & 2.93\%\,/\,100\% & 3.47\%\,/\,100\% & 2.92\%\,/\,100\% & 2.01\%\,/\,100\% \\
      & \method & 0.05\%\,/\,100\% & 0.06\%\,/\,100\% & 0.20\%\,/\,100\% & 0.12\%\,/\,100\% & 0.20\%\,/\,100\% & 0.32\%\,/\,100\% \\
      & \methoddelta & 0.09\%\,/\,100\% & 0.08\%\,/\,100\% & 0.51\%\,/\,100\% & 0.15\%\,/\,100\% & 0.23\%\,/\,100\% & 0.61\%\,/\,100\% \\
    \bottomrule
  \end{tabular}}
\end{table}

Table~\ref{tab:distributions} reports every method on the four distributions.
\methoddelta is $\alpha$-feasible on all $480$ instances under every distribution, while
retaining small gaps, and the uncorrected rule attains $0.00$--$0.32\%$ at $80$--$100\%$ feasibility.
The sampled baselines degrade as expected with SAA not finding feasible solutions on large high $\alpha$ instances,
SAA is feasible on $0$--$60\%$ of instances at $|\mathcal{S}|{=}10^3$ and on $20$--$95\%$ at $10^4$,
while CVaR stays feasible at large $|\mathcal{S}|$ but is worse in objective.  
In some cases, CVaR falls below $\alpha$ due to the sample estimate rather than the bound: the
constraint is safe for the true tail mean, but we impose it on a sample average, and at the
tightest $\alpha$ the Bonferroni split leaves too few scenarios in each row's tail to estimate it.
The mixture gains most on these distributions as well (Appendix~\ref{app:boccp_gamma_only}).

\clearpage
\section{Ablation: Nominal-Only Tightening ($\beta{=}0$)}
\label{app:boccp_gamma_only}

The scenario mixture is what separates \method from a purely parametric tightening, so we isolate its
contribution against the strongest nominal reference the formulation admits.
The $\beta{=}0$ variant runs the same pipeline, i.e., the same $\gamma$ bracket and the same selection
rule, with the constraint data fixed at its nominal value.  Nothing then varies across scenarios, so a
grid over $\gamma$ exhausts the search and we use $100$ points.
The ablation therefore measures the advantage of searching over $(\gamma, \beta)$ per sampled scenario rather than over the tightening alone.

On the Gaussian benchmarks, Table~\ref{tab:gamma_ablation} shows \method at or below $0.04\%$ FO-gap
in every cell and ahead of or level with the $\beta{=}0$ variant throughout.
Table~\ref{tab:ablation_summary} extends the comparison to the non-Gaussian distributions, where the
advantage is larger.
The nominal-only variant tightens rows by a single scalar,
so it can identify how much rows vary but not how.  In contrast, a sampled scenario is
a vector, which captures how the uncertainty affects each variable.
Under homogeneous Gaussian noise, the gains of including the scenario are marginal.
Student-$t_3$ and the clustered distribution are scaled to unit variance, so that scalar is unchanged
from the Gaussian case, yet the gain roughly doubles: heavy tails and correlation hurt the nominal
tightening even though they do not change the scalar that sizes it.
On the risky subset it reaches $+16.6\%$, where a fixed $30\%$ of the items carry three times the noise
of the rest, and a single scalar cannot distinguish the two groups.
The median $\beta^\star$ is $0.36$ on the risky subset and no instance
there sets $\beta^\star = 0$, against a median of $0.06$--$0.07$ on Gaussian knapsack and matching,
where about a third do.

\begin{table}[h!]
  \centering
  \caption{Ablation: \method vs.\ the nominal-only ($\beta{=}0$) variant, both under the
    uncertified rule.  Rows report FO-gap
    (\%) against BKO; $\Delta$ is the mean paired objective difference (positive $=$ \method ahead).}
  \label{tab:gamma_ablation}
  \small
  \begin{tabular}{ll ccc ccc ccc}
    \toprule
    & & \multicolumn{3}{c}{Small} & \multicolumn{3}{c}{Medium} & \multicolumn{3}{c}{Large} \\
    \cmidrule(lr){3-5}\cmidrule(lr){6-8}\cmidrule(lr){9-11}
    & Method & $\alpha{=}.90$ & $\alpha{=}.95$ & $\alpha{=}.99$
           & $\alpha{=}.90$ & $\alpha{=}.95$ & $\alpha{=}.99$
           & $\alpha{=}.90$ & $\alpha{=}.95$ & $\alpha{=}.99$ \\
    \midrule
    \multirow{3}{*}{\textbf{KP}}
      & \method & 0.00 & 0.01 & 0.00 & 0.00 & 0.04 & 0.00 & 0.01 & 0.01 & 0.00 \\
      & $\beta{=}0$ & 0.09 & 0.18 & 0.11 & 0.19 & 0.28 & 0.26 & 0.34 & 0.33 & 0.37 \\
      & $\Delta$ & +0.09 & +0.17 & +0.11 & +0.20 & +0.24 & +0.26 & +0.33 & +0.33 & +0.37 \\
    \midrule
    \multirow{3}{*}{\textbf{MWM}}
      & \method & 0.00 & 0.00 & 0.00 & 0.00 & 0.00 & 0.00 & 0.00 & 0.00 & 0.00 \\
      & $\beta{=}0$ & 0.19 & 0.17 & 0.43 & 0.35 & 0.21 & 0.73 & 0.26 & 0.39 & 0.66 \\
      & $\Delta$ & +0.20 & +0.17 & +0.44 & +0.35 & +0.20 & +0.74 & +0.26 & +0.40 & +0.63 \\
    \midrule
    \multirow{3}{*}{\textbf{SRCSP}}
      & \method & 0.00 & 0.00 & 0.00 & 0.03 & 0.00 & 0.00 & 0.04 & 0.00 & 0.00 \\
      & $\beta{=}0$ & 0.00 & 0.00 & 0.00 & 0.15 & 0.00 & 0.05 & 0.04 & 0.06 & 0.04 \\
      & $\Delta$ & +0.00 & +0.00 & +0.00 & +0.13 & +0.00 & +0.05 & -0.00 & +0.06 & +0.04 \\
    \bottomrule
  \end{tabular}
\end{table}

\begin{table}[h!]
  \centering
  \caption{Mean paired objective gain of \method over the nominal-only ($\beta{=}0$) variant,
    both under the uncertified rule.  Win, tie and loss count the instances on which \method is ahead,
    equal and behind.}
  \label{tab:ablation_summary}
  \small
  \begin{tabular}{lrrrrr}
    \toprule
    Benchmark & Instances & Mean gain & Win & Tie & Loss \\
    \midrule
    KP & 180 & +0.23\% & 114 & 65 & 1 \\
    MWM & 180 & +0.38\% & 67 & 113 & 0 \\
    SRCSP & 180 & +0.03\% & 5 & 174 & 1 \\
    Student-$t_3$ & 120 & +0.38\% & 104 & 14 & 2 \\
    Jump & 120 & +0.72\% & 117 & 3 & 0 \\
    Clustered & 120 & +0.81\% & 110 & 4 & 6 \\
    Risky subset & 120 & +16.60\% & 120 & 0 & 0 \\
    \bottomrule
  \end{tabular}
\end{table}

\clearpage

\section{Baseline Analysis}
\label{app:baseline_analysis}

\subsection{SAA}
\label{app:saa_analysis}

Table~\ref{tab:saa_combined} reports FO-gap and $\alpha$-feasibility rate for SAA at sample sizes $|\mathcal{S}| \in \{10, 100, 1000, 10000\}$.
Small samples hold in-sample but fail on held-out data, while large ones ($10^4$) reach the time limit and return a conservative incumbent, more often feasible but at a larger gap.

\begin{table}[h!]
  \caption{SAA results varying sample size $|\mathcal{S}|$. Each cell reports FO-gap\,/\,$\alpha$-feas rate. Time limit: 3600\,s.}
  \label{tab:saa_combined}
  \centering
  \resizebox{\textwidth}{!}{%
  \setlength\tabcolsep{3pt}
  \begin{tabular}{ll rrr rrr rrr}
    \toprule
    & & \multicolumn{3}{c}{Small}
      & \multicolumn{3}{c}{Medium}
      & \multicolumn{3}{c}{Large} \\
    \cmidrule(lr){3-5}\cmidrule(lr){6-8}\cmidrule(lr){9-11}
    & $|\mathcal{S}|$
      & $\alpha{=}.90$ & $\alpha{=}.95$ & $\alpha{=}.99$
      & $\alpha{=}.90$ & $\alpha{=}.95$ & $\alpha{=}.99$
      & $\alpha{=}.90$ & $\alpha{=}.95$ & $\alpha{=}.99$ \\
    \midrule
    \multirow{4}{*}{\textbf{KP}}
      & 10 & ---\,/\,0\% & ---\,/\,0\% & ---\,/\,0\% & ---\,/\,0\% & ---\,/\,0\% & ---\,/\,0\% & ---\,/\,0\% & ---\,/\,0\% & ---\,/\,0\% \\
      & 100 & 0.14\%\,/\,20\% & 0.00\%\,/\,10\% & ---\,/\,0\% & ---\,/\,0\% & ---\,/\,0\% & ---\,/\,0\% & ---\,/\,0\% & ---\,/\,0\% & ---\,/\,0\% \\
      & 1000 & 0.08\%\,/\,75\% & 0.11\%\,/\,50\% & 0.00\%\,/\,45\% & 0.21\%\,/\,40\% & 0.19\%\,/\,25\% & ---\,/\,0\% & 0.11\%\,/\,25\% & 0.29\%\,/\,5\% & ---\,/\,0\% \\
      & 10000 & 0.10\%\,/\,85\% & 0.10\%\,/\,70\% & 0.04\%\,/\,80\% & 0.12\%\,/\,90\% & 0.04\%\,/\,75\% & 0.03\%\,/\,30\% & 0.19\%\,/\,95\% & 0.14\%\,/\,70\% & 0.06\%\,/\,30\% \\
    \midrule
    \multirow{4}{*}{\textbf{MWM}}
      & 10 & ---\,/\,0\% & ---\,/\,0\% & ---\,/\,0\% & ---\,/\,0\% & ---\,/\,0\% & ---\,/\,0\% & ---\,/\,0\% & ---\,/\,0\% & ---\,/\,0\% \\
      & 100 & 0.41\%\,/\,45\% & 0.00\%\,/\,35\% & 0.00\%\,/\,15\% & 0.13\%\,/\,15\% & 0.00\%\,/\,5\% & ---\,/\,0\% & 0.55\%\,/\,10\% & ---\,/\,0\% & ---\,/\,0\% \\
      & 1000 & 0.09\%\,/\,75\% & 0.17\%\,/\,75\% & 0.13\%\,/\,55\% & 0.80\%\,/\,85\% & 0.62\%\,/\,55\% & 0.11\%\,/\,55\% & 1.38\%\,/\,65\% & 0.84\%\,/\,80\% & 0.04\%\,/\,40\% \\
      & 10000 & 3.49\%\,/\,95\% & 3.08\%\,/\,95\% & 0.52\%\,/\,95\% & 4.26\%\,/\,95\% & 4.28\%\,/\,100\% & 1.70\%\,/\,95\% & 5.23\%\,/\,95\% & 4.57\%\,/\,100\% & 1.84\%\,/\,95\% \\
    \midrule
    \multirow{4}{*}{\textbf{SRCSP}}
      & 10 & 0.00\%\,/\,55\% & 2.09\%\,/\,50\% & 0.00\%\,/\,20\% & 2.31\%\,/\,15\% & 1.68\%\,/\,20\% & ---\,/\,0\% & ---\,/\,0\% & 0.00\%\,/\,15\% & ---\,/\,0\% \\
      & 100 & 0.00\%\,/\,90\% & 0.13\%\,/\,90\% & 0.00\%\,/\,80\% & 0.12\%\,/\,85\% & 1.17\%\,/\,80\% & 0.39\%\,/\,45\% & 0.11\%\,/\,55\% & 0.18\%\,/\,40\% & 0.00\%\,/\,15\% \\
      & 1000 & 1.84\%\,/\,95\% & 0.00\%\,/\,100\% & 0.00\%\,/\,95\% & 0.00\%\,/\,90\% & 0.98\%\,/\,95\% & 0.79\%\,/\,70\% & 0.07\%\,/\,85\% & 0.28\%\,/\,75\% & 0.41\%\,/\,55\% \\
      & 10000 & 1.74\%\,/\,100\% & 0.00\%\,/\,95\% & 0.00\%\,/\,100\% & 0.52\%\,/\,100\% & 0.09\%\,/\,100\% & 0.00\%\,/\,85\% & 3.58\%\,/\,95\% & 2.09\%\,/\,100\% & 0.69\%\,/\,95\% \\
    \bottomrule
  \end{tabular}}
\end{table}

\subsection{CVaR}
\label{app:cvar_analysis}

Table~\ref{tab:cvar_combined} reports the same for CVaR, which adds only continuous auxiliary variables and so reaches the hour far less often than SAA.
Small samples leave one or two scenarios in each row's tail, so the estimate is not safe and $\alpha{=}0.99$ fails even at $|\mathcal{S}|{=}10^3$ scenarios. 
Moreover, the Bonferroni allocation leads to overly conservative solutions, and the degradation relative to SAA is significant.  

\begin{table}[h!]
  \caption{CVaR results varying sample size $|\mathcal{S}|$. Each cell reports FO-gap\,/\,$\alpha$-feas rate. Time limit: 3600\,s.}
  \label{tab:cvar_combined}
  \centering
  \resizebox{\textwidth}{!}{%
  \setlength\tabcolsep{3pt}
  \begin{tabular}{ll rrr rrr rrr}
    \toprule
    & & \multicolumn{3}{c}{Small}
      & \multicolumn{3}{c}{Medium}
      & \multicolumn{3}{c}{Large} \\
    \cmidrule(lr){3-5}\cmidrule(lr){6-8}\cmidrule(lr){9-11}
    & $|\mathcal{S}|$
      & $\alpha{=}.90$ & $\alpha{=}.95$ & $\alpha{=}.99$
      & $\alpha{=}.90$ & $\alpha{=}.95$ & $\alpha{=}.99$
      & $\alpha{=}.90$ & $\alpha{=}.95$ & $\alpha{=}.99$ \\
    \midrule
    \multirow{4}{*}{\textbf{KP}}
      & 10 & ---\,/\,0\% & ---\,/\,0\% & ---\,/\,0\% & 0.00\%\,/\,5\% & ---\,/\,0\% & ---\,/\,0\% & ---\,/\,0\% & ---\,/\,0\% & ---\,/\,0\% \\
      & 100 & 4.13\%\,/\,100\% & 2.59\%\,/\,80\% & 0.00\%\,/\,5\% & 2.47\%\,/\,90\% & 1.52\%\,/\,30\% & ---\,/\,0\% & 1.33\%\,/\,60\% & ---\,/\,0\% & ---\,/\,0\% \\
      & 1000 & 5.82\%\,/\,100\% & 4.88\%\,/\,100\% & 3.33\%\,/\,100\% & 4.16\%\,/\,100\% & 3.40\%\,/\,100\% & 1.37\%\,/\,85\% & 3.25\%\,/\,100\% & 2.45\%\,/\,100\% & 0.89\%\,/\,35\% \\
      & 10000 & 6.00\%\,/\,100\% & 5.26\%\,/\,100\% & 4.39\%\,/\,100\% & 4.51\%\,/\,100\% & 4.02\%\,/\,100\% & 3.12\%\,/\,100\% & 3.81\%\,/\,100\% & 3.34\%\,/\,100\% & 2.38\%\,/\,100\% \\
    \midrule
    \multirow{4}{*}{\textbf{MWM}}
      & 10 & 0.00\%\,/\,10\% & ---\,/\,0\% & ---\,/\,0\% & ---\,/\,0\% & ---\,/\,0\% & ---\,/\,0\% & ---\,/\,0\% & ---\,/\,0\% & ---\,/\,0\% \\
      & 100 & 6.59\%\,/\,100\% & 4.11\%\,/\,95\% & 0.15\%\,/\,25\% & 3.57\%\,/\,100\% & 2.42\%\,/\,80\% & ---\,/\,0\% & 2.60\%\,/\,100\% & 1.58\%\,/\,30\% & ---\,/\,0\% \\
      & 1000 & 9.20\%\,/\,100\% & 6.76\%\,/\,100\% & 5.15\%\,/\,100\% & 5.75\%\,/\,100\% & 5.18\%\,/\,100\% & 3.39\%\,/\,95\% & 4.49\%\,/\,100\% & 3.85\%\,/\,100\% & 1.94\%\,/\,100\% \\
      & 10000 & 9.49\%\,/\,100\% & 6.43\%\,/\,100\% & 6.57\%\,/\,100\% & 6.75\%\,/\,100\% & 6.42\%\,/\,100\% & 4.77\%\,/\,100\% & 6.61\%\,/\,100\% & 6.27\%\,/\,100\% & 4.68\%\,/\,100\% \\
    \midrule
    \multirow{4}{*}{\textbf{SRCSP}}
      & 10 & 4.72\%\,/\,80\% & 2.09\%\,/\,50\% & 0.00\%\,/\,20\% & 1.66\%\,/\,60\% & 1.68\%\,/\,20\% & ---\,/\,0\% & 0.42\%\,/\,40\% & 0.00\%\,/\,15\% & ---\,/\,0\% \\
      & 100 & 10.97\%\,/\,100\% & 7.40\%\,/\,100\% & 1.50\%\,/\,90\% & 7.62\%\,/\,100\% & 5.91\%\,/\,100\% & 1.74\%\,/\,75\% & 6.00\%\,/\,100\% & 3.77\%\,/\,100\% & 0.00\%\,/\,35\% \\
      & 1000 & 10.38\%\,/\,100\% & 7.16\%\,/\,100\% & 3.10\%\,/\,100\% & 8.56\%\,/\,100\% & 6.84\%\,/\,100\% & 6.95\%\,/\,100\% & 8.62\%\,/\,100\% & 7.44\%\,/\,100\% & 4.62\%\,/\,100\% \\
      & 10000 & 11.31\%\,/\,100\% & 7.86\%\,/\,100\% & 7.92\%\,/\,100\% & 8.17\%\,/\,100\% & 7.59\%\,/\,100\% & 6.73\%\,/\,100\% & 8.89\%\,/\,100\% & 7.93\%\,/\,100\% & 6.47\%\,/\,100\% \\
    \bottomrule
  \end{tabular}}
\end{table}

\clearpage
\section{Per-Instance Distributional Results}
\label{app:instance_distributions}

This section reports the distribution of outcomes across the 20 instances for each benchmark, instance size, and $\alpha$ value. We compare \method against the baseline methods: SAA ($|\mathcal{S}|{=}1000$), CVaR ($|\mathcal{S}|{=}1000$), and the GB approximation.

\subsection{Feasibility Rate vs. Optimality Gap}

Figures~\ref{fig:scatter_grid_kp}--\ref{fig:scatter_grid_srcsp} plot feasibility rate against optimality gap, with columns varying $\alpha$ and rows varying the problem size, and the dashed line marking the $\alpha$ threshold.
Every solution is shown, feasible or not, and the gap is measured against the best $\alpha$-feasible objective found by any method, the reference of Table~\ref{tab:results}, so an infeasible solution that beats it appears at a negative gap.
The three benchmarks show the same pattern: \method sits near $0\%$ gap just above the threshold, SAA falls below it, and CVaR and GB clear it at a substantially larger gap.

\begin{figure}[h!]
    \centering
    \includegraphics[width=\textwidth]{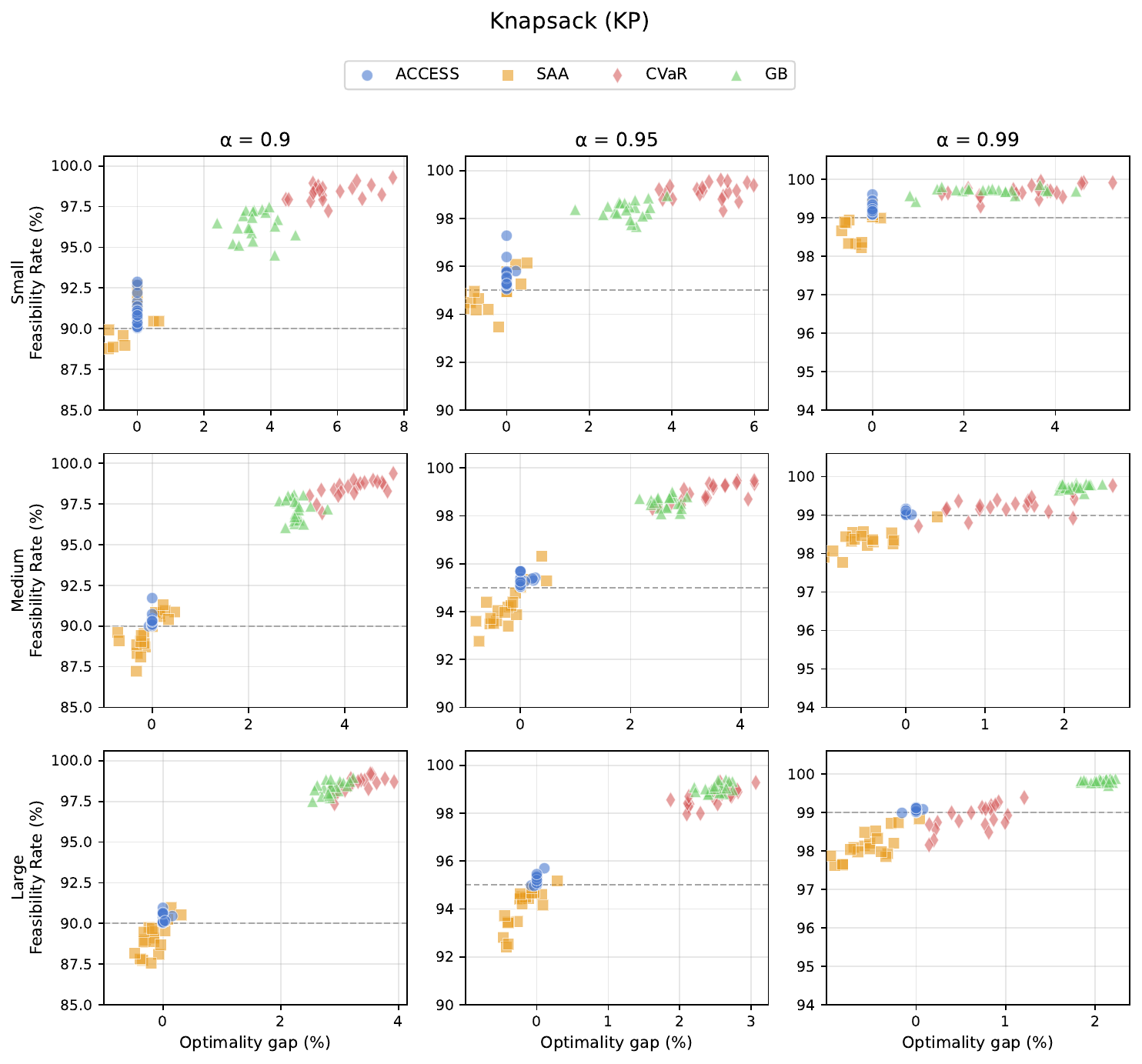}
    \caption{Per-instance feasibility rate versus optimality gap for the Knapsack Problem (KP).  Each point is one of the 20 instances, feasible or not.}
    \label{fig:scatter_grid_kp}
\end{figure}

\begin{figure}[h!]
    \centering
    \includegraphics[width=\textwidth]{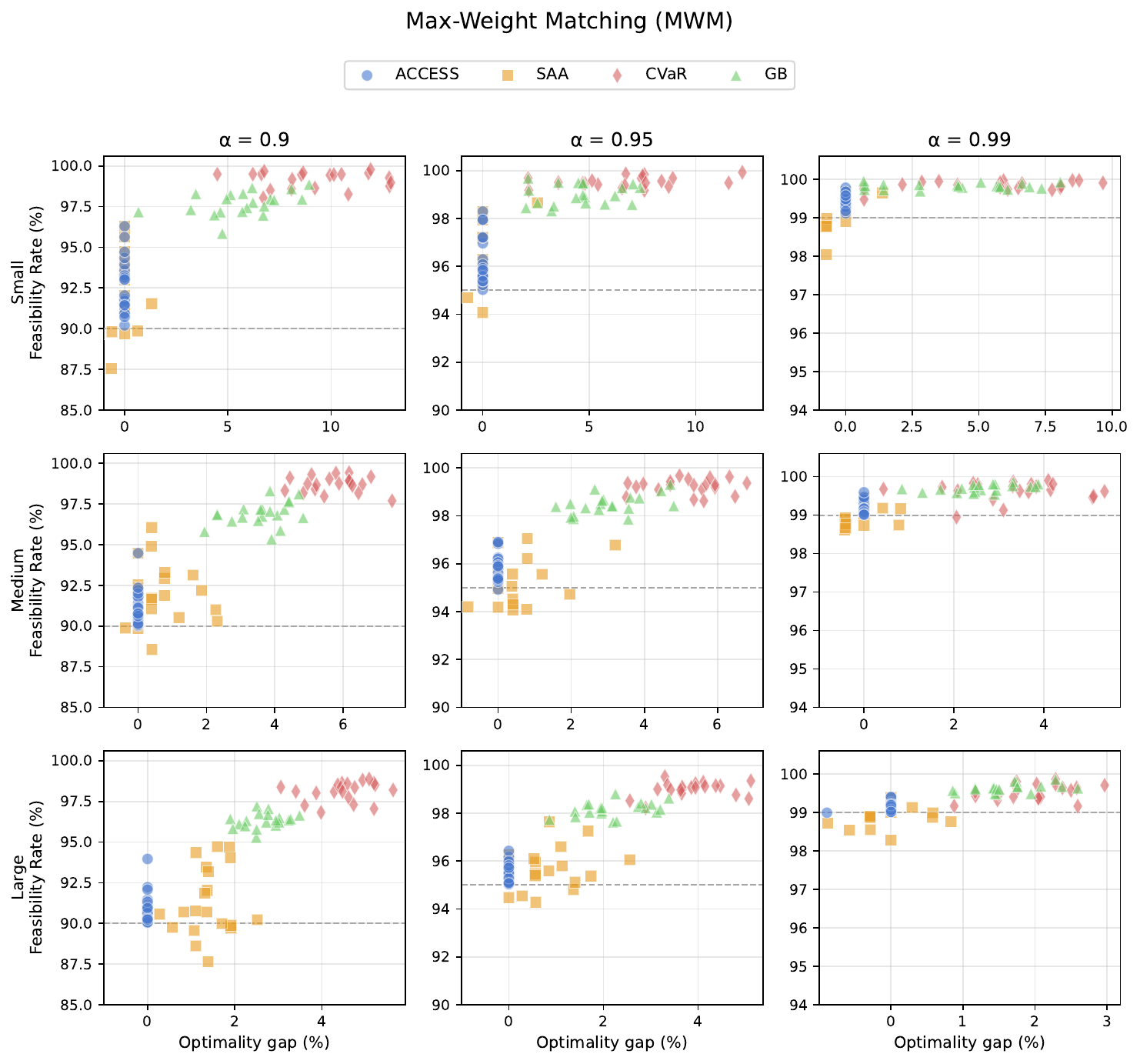}
    \caption{Per-instance feasibility rate versus optimality gap for Maximum Weight Matching (MWM).}
    \label{fig:scatter_grid_mwm}
\end{figure}

\begin{figure}[h!]
    \centering
    \includegraphics[width=\textwidth]{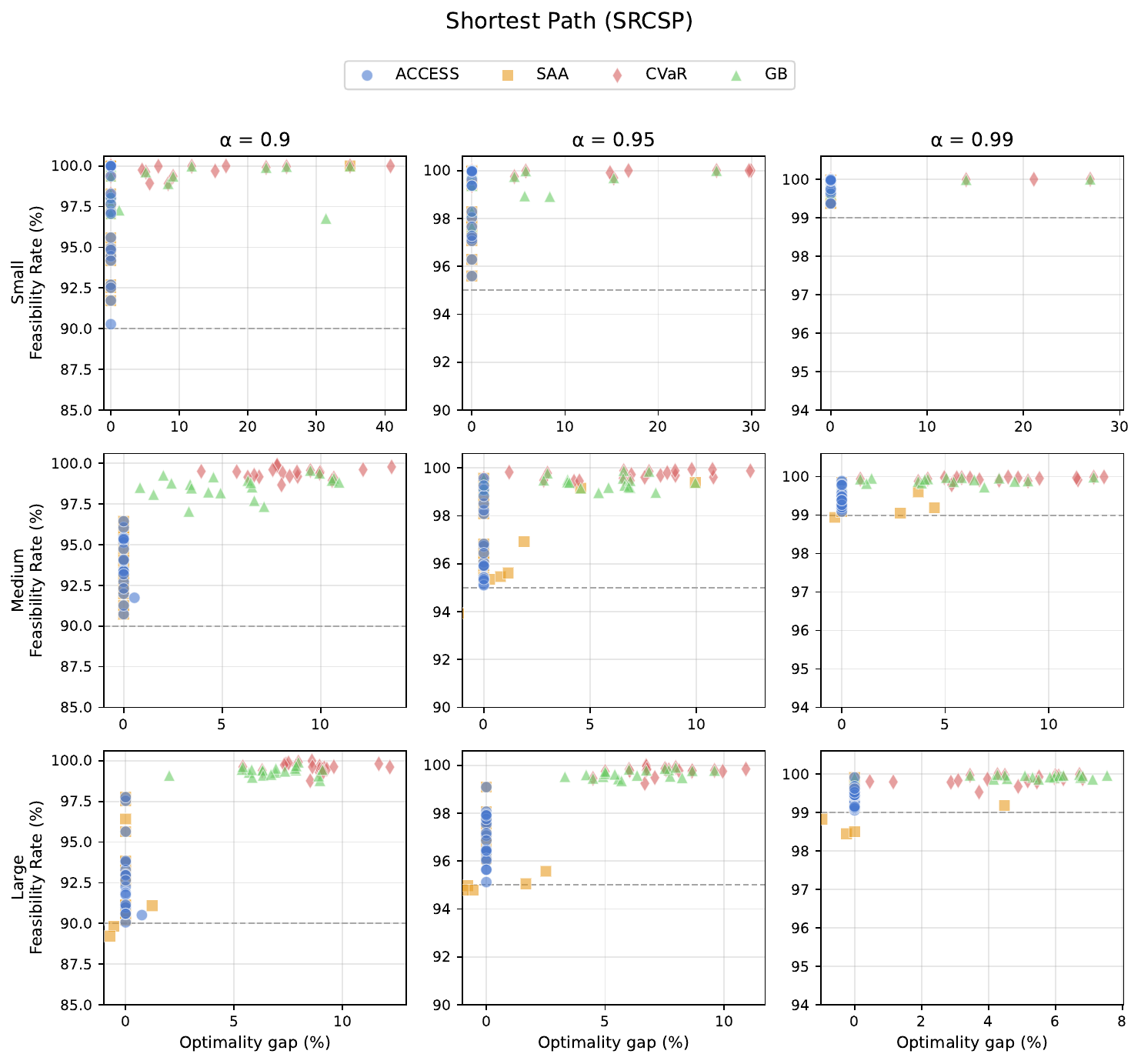}
    \caption{Per-instance feasibility rate versus optimality gap for Stochastic Resource-Constrained Shortest Path (SRCSP).}
    \label{fig:scatter_grid_srcsp}
\end{figure}

\subsection{Generation Time vs. FO-gap}

Figures~\ref{fig:time_fogap_kp}--\ref{fig:time_fogap_srcsp} compare time against FO-gap, and include only $\alpha$-feasible solutions, since runtime is not comparable for a solution that does not meet the constraint.
Again the three benchmarks agree: \method reaches near-zero FO-gap in about two minutes on KP and SRCSP and about ten on the largest MWM instances, SAA reaches the hour, and GB and CVaR are faster but pay a much larger gap.

\begin{figure}[h!]
    \centering
    \includegraphics[width=\textwidth]{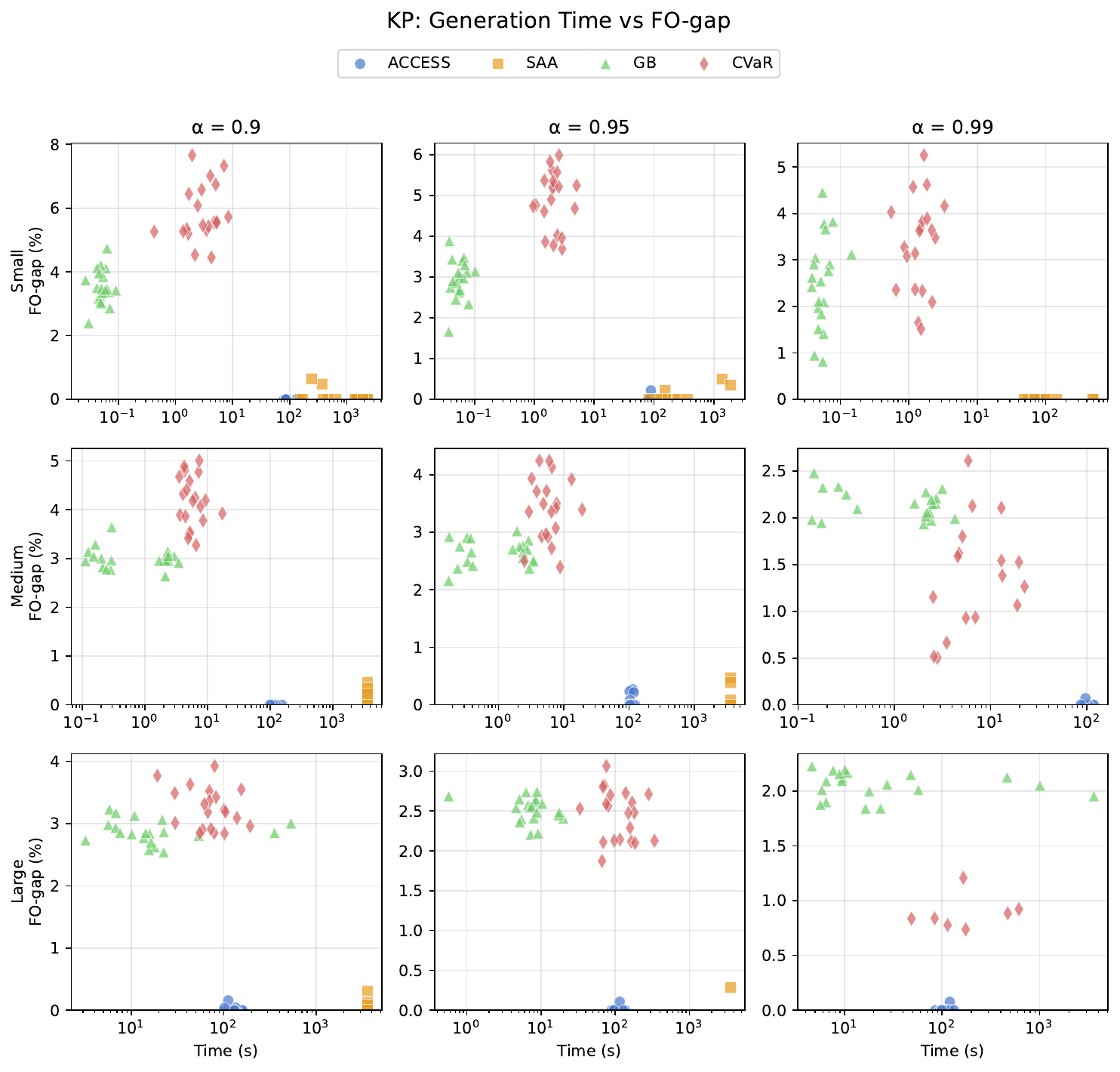}
    \caption{Generation time versus FO-gap for KP, over the $\alpha$-feasible solutions only.}
    \label{fig:time_fogap_kp}
\end{figure}

\begin{figure}[h!]
    \centering
    \includegraphics[width=\textwidth]{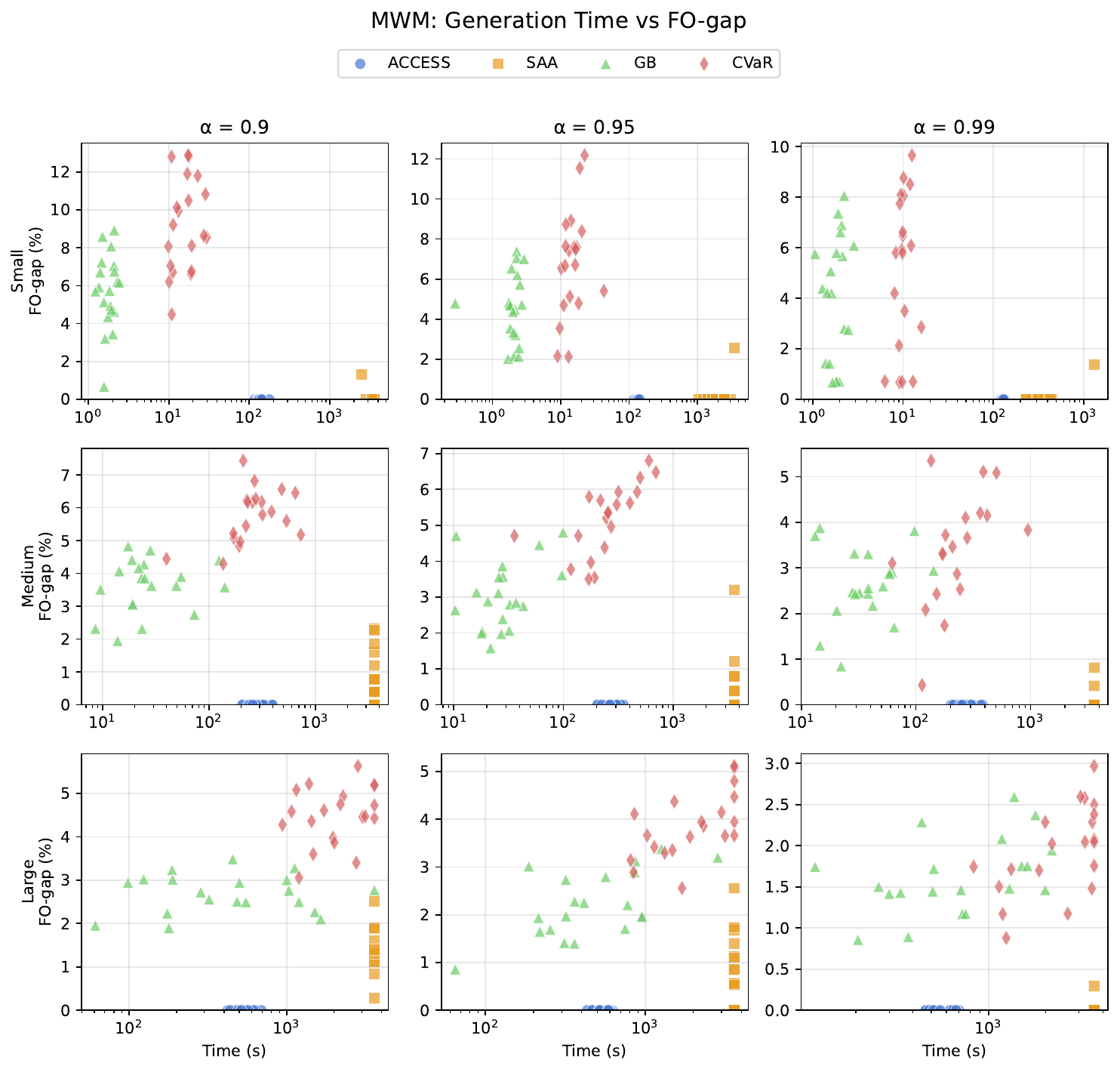}
    \caption{Generation time versus FO-gap for MWM, over the $\alpha$-feasible solutions only.}
    \label{fig:time_fogap_mwm}
\end{figure}

\begin{figure}[h!]
    \centering
    \includegraphics[width=\textwidth]{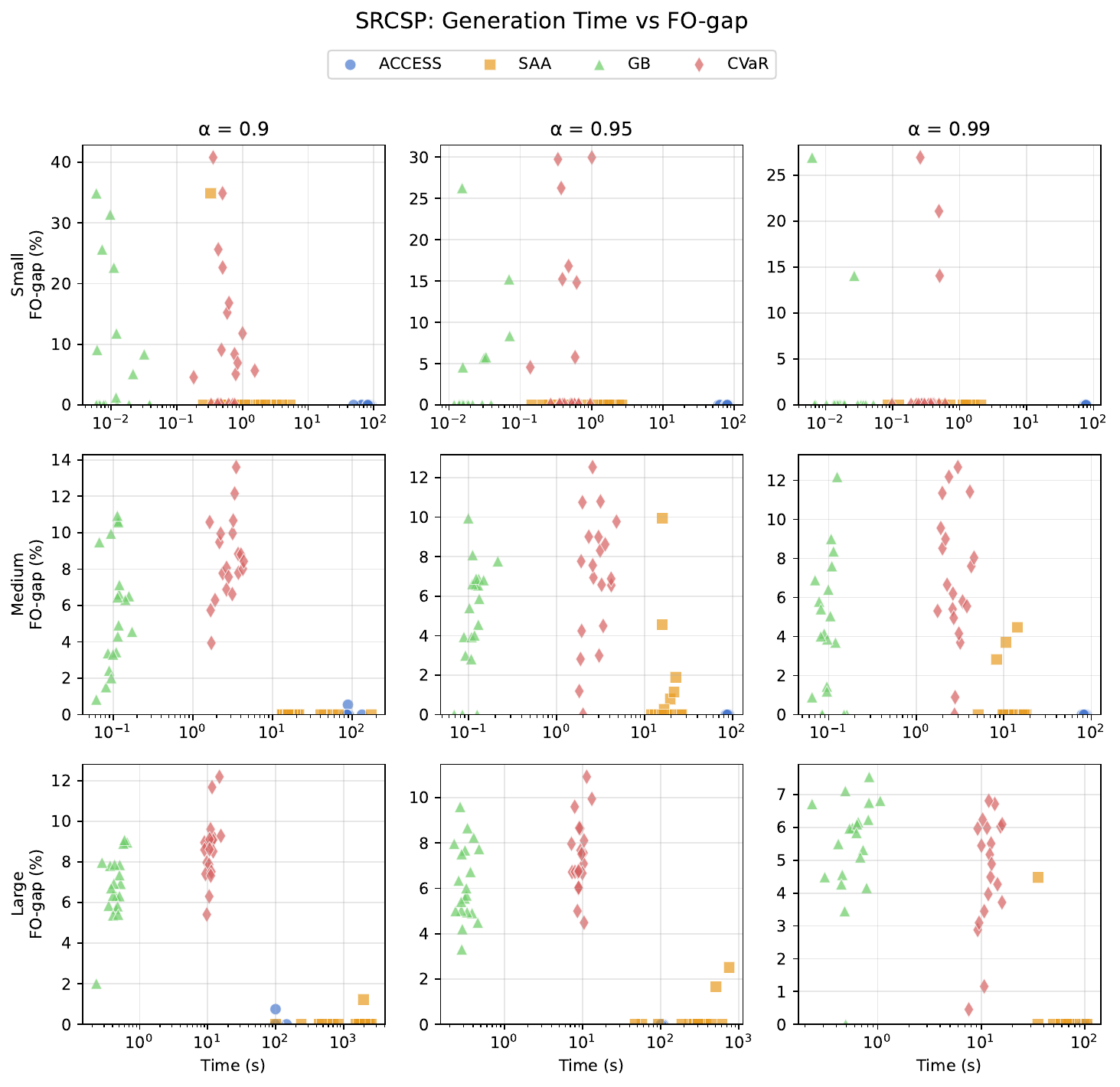}
    \caption{Generation time versus FO-gap for SRCSP, over the $\alpha$-feasible solutions only.}
    \label{fig:time_fogap_srcsp}
\end{figure}

\subsection{Feasibility Distributions: \method vs Random Sampling }
\label{app:feas_distributions} 

Figures~\ref{fig:feasibility_distributions_kp}--\ref{fig:feasibility_distributions_srcsp} compare
the per-candidate feasibility (evaluated on $10^6$ scenarios) of the \method pool with that of
$1{,}000$ random solutions per instance, on every instance size and $\alpha$, pooled across the
20 test instances.  Consistently, we can observe that \method concentrates solutions near the boundary, while random sampling often produces completely feasible or infeasible solutions.  

\begin{figure}[h!]
    \centering
    \includegraphics[width=\textwidth]{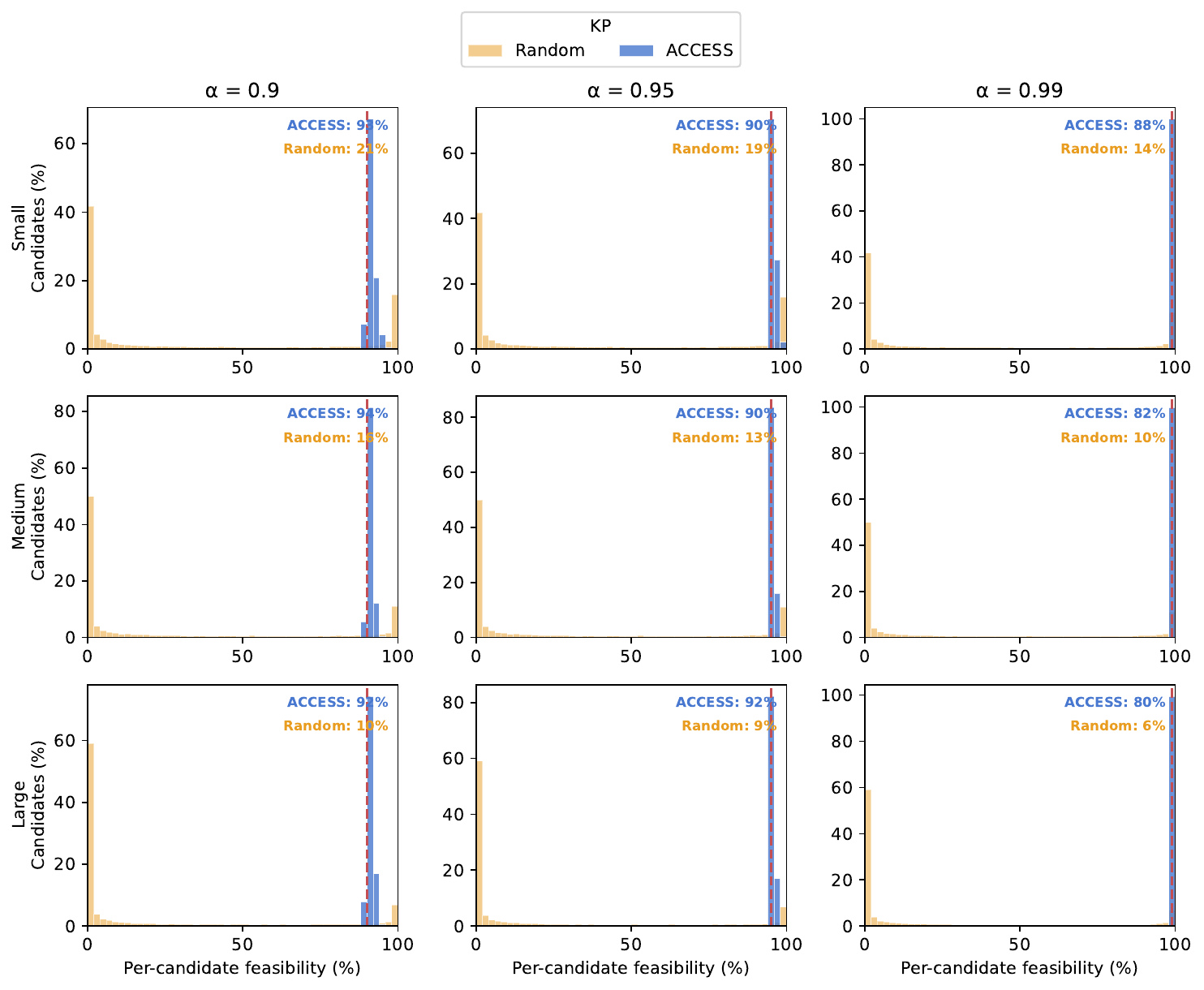}
    \caption{Feasibility distributions for Knapsack (KP).  Rows are instance sizes and columns are $\alpha$.}
    \label{fig:feasibility_distributions_kp}
\end{figure}

\begin{figure}[h!]
    \centering
    \includegraphics[width=\textwidth]{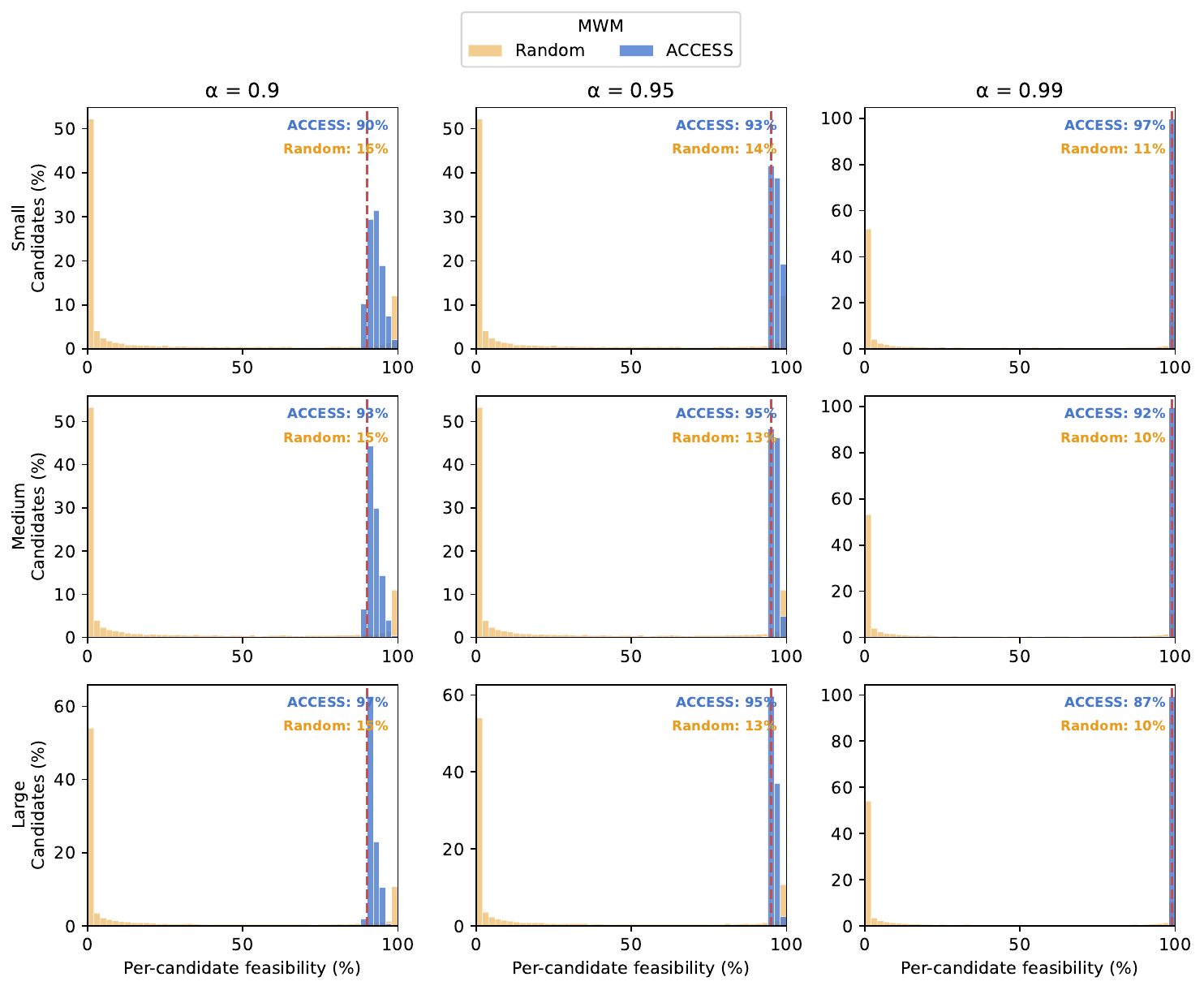}
    \caption{Feasibility distributions for Maximum Weight Matching (MWM).}
    \label{fig:feasibility_distributions_mwm}
\end{figure}

\begin{figure}[h!]
    \centering
    \includegraphics[width=\textwidth]{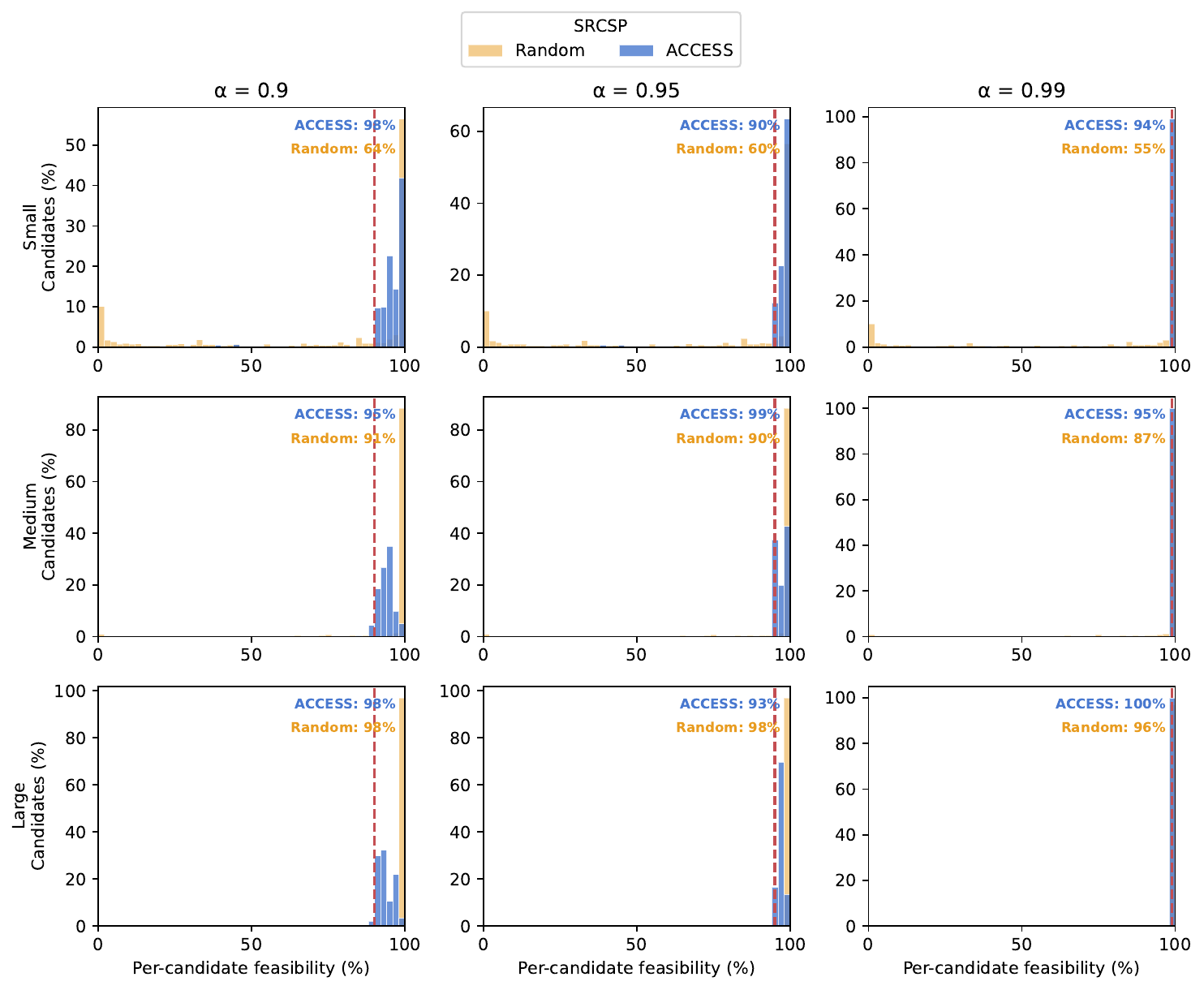}
    \caption{Feasibility distributions for Stochastic Resource-Constrained Shortest Path (SRCSP).}
    \label{fig:feasibility_distributions_srcsp}
\end{figure}

\clearpage
\section{Detailed Runtime Results}
\label{app:detailed_runtime}

Table~\ref{tab:runtime_detailed} provides detailed generation time statistics (mean across 20 test instances) for each benchmark, size, and $\alpha$ value. For \method, times are shown as generation/feasibility to separate the BO search from the final constraint verification.
Feasibility checking is $2$--$13\%$ of the total, so a larger evaluation set, and with it a tighter certificate, costs little.

\begin{table}[h!]
  \centering
  \caption{Generation time in seconds (mean across 20 instances).
    \method times shown as generation/feasibility.}
  \label{tab:runtime_detailed}
  \small
  \resizebox{\textwidth}{!}{%
  \setlength\tabcolsep{3pt}
  \begin{tabular}{ll rrr rrr rrr}
    \toprule
    & & \multicolumn{3}{c}{Small}
      & \multicolumn{3}{c}{Medium}
      & \multicolumn{3}{c}{Large} \\
    \cmidrule(lr){3-5}\cmidrule(lr){6-8}\cmidrule(lr){9-11}
    & Method & $\alpha{=}.90$ & $\alpha{=}.95$ & $\alpha{=}.99$
      & $\alpha{=}.90$ & $\alpha{=}.95$ & $\alpha{=}.99$
      & $\alpha{=}.90$ & $\alpha{=}.95$ & $\alpha{=}.99$ \\
    \midrule
    \multirow{4}{*}{\textbf{KP}}
      & \method & 87.3/2.7 & 85.7/2.8 & 81.6/2.8 & 111.1/2.7 & 109.8/2.8 & 103.2/2.8 & 122.9/16.1 & 109.6/15.9 & 106.0/15.8 \\
      & SAA & 1082 & 440 & 199 & 3601 & 3601 & 3561 & 3600 & 3600 & 3600 \\
      & GB & 0.1 & 0.1 & 0.1 & 1.3 & 1.6 & 1.7 & 58 & 8.6 & 267 \\
      & CVaR & 3.5 & 2.3 & 1.6 & 6.3 & 6.6 & 9.1 & 81 & 130 & 256 \\
    \midrule
    \multirow{4}{*}{\textbf{MWM}}
      & \method & 131.3/6.7 & 130.7/6.4 & 127.0/6.7 & 281.6/15.3 & 276.5/14.8 & 271.3/14.4 & 526.7/22.9 & 543.2/23.6 & 557.3/22.4 \\
      & SAA & 3397 & 2082 & 350 & 3600 & 3600 & 3600 & 3600 & 3600 & 3600 \\
      & GB & 1.8 & 2.1 & 1.9 & 37 & 34 & 45 & 735 & 642 & 906 \\
      & CVaR & 17 & 16 & 10 & 301 & 289 & 268 & 2249 & 2245 & 2622 \\
    \midrule
    \multirow{4}{*}{\textbf{SRCSP}}
      & \method & 81.4/2.5 & 79.3/2.5 & 75.6/2.6 & 89.7/3.5 & 87.7/3.2 & 85.2/3.3 & 103.0/9.0 & 99.9/8.9 & 96.7/9.1 \\
      & SAA & 1.5 & 1.0 & 0.8 & 32 & 18 & 12 & 1293 & 294 & 66 \\
      & GB & 0.0 & 0.0 & 0.0 & 0.1 & 0.1 & 0.1 & 0.5 & 0.3 & 0.6 \\
      & CVaR & 0.6 & 0.5 & 0.3 & 2.9 & 2.9 & 2.9 & 11 & 9.4 & 12 \\
    \bottomrule
  \end{tabular}}
\end{table}

\end{document}